\documentclass[11pt]{article}%
\usepackage{amsxtra,amscd}
\usepackage{amsmath}
\usepackage{amsfonts}
\usepackage{amssymb}
\usepackage{graphicx}%
\usepackage[textwidth=6.0in,textheight=9.0in,centering]{geometry}
\usepackage[amsmath,hyperref,thmmarks]{ntheorem}
\usepackage{natbib}
\usepackage{times}
\usepackage{paralist}
\usepackage[all,cmtip,line]{xy}%

\def\1{{1\mkern-7mu1}}

\newcommand{\tableoc}{\tableofcontents}
\newcommand\bquote{\begin{quote}}
\newcommand\equote{\end{quote}}
\newcommand\bsmall{\begin{small}}
\newcommand\esmall{\end{small}}
\newcommand\bfootnotesize{\begin{footnotesize}}
\newcommand\efootnotesize{\end{footnotesize}}

\renewcommand{\theenumi}{{\alph{enumi}}}

\renewcommand{\theenumii}{\roman{enumii}}

\let\cite=\citealt
\let\nocite=\citet

\newcommand{\eb}[1]{{\itshape\bfseries#1}{\index{#1}}}
\renewcommand{\emph}{\eb}

\renewcommand{\Gamma}{\varGamma}
\renewcommand{\Pi}{\varPi}
\renewcommand{\Sigma}{\varSigma}

\DeclareMathOperator{\Aut}{Aut}

\DeclareMathOperator{\End}{End}

\DeclareMathOperator{\Ext}{Ext}

\DeclareMathOperator{\Gal}{Gal}

\DeclareMathOperator{\Hom}{Hom}
\DeclareMathOperator{\id}{id}

\DeclareMathOperator{\Ker}{Ker}

\DeclareMathOperator{\Nm}{Nm}

\DeclareMathOperator{\ord}{ord}

\DeclareMathOperator{\plim}{\varprojlim}

\DeclareMathOperator{\rec}{rec}

\DeclareMathOperator{\Spec}{Spec}

\DeclareMathOperator{\Tgt}{Tgt}

\DeclareMathOperator{\Tr}{Tr}

\newcommand{\ab}{\mathrm{ab}}
\newcommand{\ad}{\mathrm{ad}}
\newcommand{\al}{\mathrm{al}}

\makeindex
\theoremnumbering{arabic}
\theoremheaderfont{\scshape}
\RequirePackage{latexsym}
\theorembodyfont{\slshape}
\theoremseparator{}
\newtheorem{X}{X}[section]

\newtheorem{corollary}[X]{Corollary}

\newtheorem{lemma}[X]{Lemma}
\newtheorem{proposition}[X]{Proposition}
\newtheorem{theorem}[X]{Theorem}
\theorembodyfont{\upshape}
\newtheorem{aside}[X]{Aside}
\newtheorem{definition}[X]{Definition}
\newtheorem{example}[X]{Example}

\newtheorem{remark}[X]{Remark}
\newtheorem{plain}[X]{}

\newtheorem{application}[X]{Application}
\theorembodyfont{\small}
\newtheorem*{nt}{Notes}
\theoremstyle{nonumberplain}
\theorembodyfont{\normalsize}
\theoremsymbol{\ensuremath{_\Box}}
\RequirePackage{amssymb}
\newtheorem{proof}{Proof.}
\newtheorem{pf}{Proof}
\qedsymbol{\ensuremath{_\Box}}
\theoremclass{LaTeX}
\begin{document}

\title{The Fundamental Theorem of Complex Multiplication \thanks{\noindent Copyright
\copyright \thinspace\ 2006, 2007. J.S. Milne. }}
\author{J.S. Milne}
\date{May 23, 2007}
\maketitle

\begin{abstract}
The goal of this expository article is to present a proof that is as direct
and elementary as possible of the fundamental theorem of complex
multiplication (Shimura, Taniyama, Langlands, Tate, Deligne et al.).

The article is a revision of part of my manuscript Milne 2006.

\end{abstract}

\pagestyle{headings}

\tableoc

\subsection*{Introduction}

A simple abelian variety $A$ over $\mathbb{C}{}$ is said to have complex
multiplication if its endomorphism algebra is a field of degree $2\dim A$ over
$\mathbb{Q}{}$, and a general abelian variety over $\mathbb{C}{}$ is said to
have complex multiplication if each of its simple isogeny factors does.
Abelian varieties with complex multiplication correspond to special points in
the moduli space of abelian varieties, and their arithmetic is intimately
related to that of the values of modular functions and modular forms at those
points. The fundamental theorem of complex multiplication describes how an
automorphism of $\mathbb{C}{}$ (as an abstract field) acts on the abelian
varieties with complex multiplication and their torsion points.

The basic theory of complex multiplication was extended from elliptic curves
to abelian varieties in the 1950s by Shimura, Taniyama, and Weil.\footnote{See
the articles by Shimura, Taniyama, and Weil in: Proceedings of the
International Symposium on Algebraic Number Theory, Tokyo \& Nikko, September,
1955. Science Council of Japan, Tokyo, 1956.} The first result in the
direction of the fundamental theorem is the formula of Taniyama for the
prime-ideal decomposition of an endomorphism of an abelian variety that
becomes the Frobenius map modulo $p$.\footnote{Ibid. p21 (article of Weil).}
In their book \cite{shimuraT1961}, and in various other works, Shimura and
Taniyama proved the fundamental theorem for automorphisms of $\mathbb{C}{}$
fixing the reflex field of the abelian variety. Except for the result of
\cite{shih1976}, no progress was made on the problem of extending the theorem
to all automorphisms of $\mathbb{C}$ until the article \cite{langlands1979c}.
In that work, Langlands attempted to understand how the automorphisms of
$\mathbb{C}{}$ act on Shimura varieties and their special points, and in doing
so he was led to define a cocycle that conjecturally describes how the
automorphisms of $\mathbb{C}{}$ act on abelian varieties with complex
multiplication and their torsion points. Langlands's cocycle enables one to
give a precise conjectural statement of the fundamental theorem \textit{over}
$\mathbb{Q}{}$. \nocite{tate1981} gave a more elementary construction of
Langlands's cocycle and he proved that it did indeed describe the action of
$\Aut(\mathbb{C})$ on abelian varieties of CM-type and their torsion points up
to a sequence of signs indexed by the primes of $\mathbb{Q}$. Finally,
\cite{deligne1982m} showed that there exists at most one cocycle describing
this action of $\Aut(\mathbb{C}{})$ that is consistent with the results of
Shimura and Taniyama, and so completed the proof of the fundamental theorem
over $\mathbb{Q}{}$.

The goal of this article is to present a proof of the fundamental theorem of
complex multiplication that is as direct and elementary as possible.

I assume that the reader is familiar with some of the more elementary parts of
the theory of complex multiplication. See \cite{milneCM} for more background.

\subsubsection{Notations.}

\textquotedblleft Field\textquotedblright\ means \textquotedblleft commutative
field\textquotedblright, and \textquotedblleft number field\textquotedblright%
\ means \textquotedblleft field of finite degree over $\mathbb{Q}{}%
$\textquotedblright\ (not necessarily contained in $\mathbb{C}{}$). The ring
of integers in a number field $k$ is denoted by $\mathcal{O}{}_{k}$, and
$k^{\mathrm{al}}$ denotes an algebraic closure of a field $k$. By
$\mathbb{C}{}$, I mean an algebraic closure of $\mathbb{R}{}$ , and
$\mathbb{Q}{}^{\mathrm{al}}$ is the algebraic closure of $\mathbb{Q}{}$ in
$\mathbb{C}{}$. Complex conjugation on $\mathbb{C}{}$ (or a subfield) is
denoted by $\iota$.

For an abelian group $X$ and integer $m$, $X_{m}=\{x\in X\mid mx=0\}$.

An \'{e}tale algebra over a field is a finite product of finite separable
field extensions of the field. When $E$ is an \'{e}tale $\mathbb{Q}{}$-algebra
and $k$ is a field of characteristic zero, I say that $k$ \emph{contains all
conjugates of} $E$ when every $\mathbb{Q}{}$-algebra homomorphism
$E\rightarrow k^{\mathrm{al}}$ maps into $k$. This means that there are
exactly $[E\colon\mathbb{Q}{}]$ distinct $\mathbb{Q}{}$-algebra homomorphisms
$E\rightarrow k$.

Rings are required to have a $1$, homomorphisms of rings are required to map
$1$ to $1$, and $1$ is required to act as the identity map on any module. By a
$k$-algebra ($k$ a field) I mean a ring $B$ containing $k$ in its centre.

Following \cite{bourbakiTG}, I \S 9.1, I require compact topological spaces to
be separated.

\section{Preliminaries}

\subsection{CM-algebras; CM-types; reflex norms}

A number field $E$ is said to be a \emph{CM-field} if there exists an
automorphism $\iota_{E}\neq1$ of $E$ such that $\rho\circ\iota_{E}=\iota
\circ\rho$ for every embedding $\rho{}$ of $E$ into $\mathbb{C}{}$.
Equivalently, $E=F[\sqrt{a}]$ with $F$ a totally real number field and $a$ a
totally negative element of $F$. A \emph{CM-algebra} is a finite product of CM-fields.

For a CM-algebra $E$ the homomorphisms $E\rightarrow\mathbb{C}{}$ occur in
complex conjugate pairs $\{\varphi,\iota\circ\varphi\}$. A \emph{CM-type} on
$E$ is a choice of one element from each pair. More formally, it is a subset
$\Phi$ of $\Hom(E,\mathbb{C}{})$ such that%
\[
\Hom(E,\mathbb{C}{})=\Phi\sqcup\iota\Phi\quad\quad\text{(disjoint union).}%
\]
A \emph{CM-pair} is a CM-algebra together with a CM-type.

Let $(E,\Phi)$ be a CM-pair, and for $a\in E$, let $\Tr_{\Phi}(a)=\sum
_{\varphi\in\Phi}\varphi(a)\in\mathbb{C}{}$. The \emph{reflex field} $E^{\ast
}$ of $(E,\Phi)$ is the subfield of $\mathbb{C}{}$ generated by the elements
$\Tr_{\Phi}(a)$, $a\in E$. It can also be described as the fixed field of
$\{\sigma\in\Gal(\mathbb{Q}{}^{\mathrm{al}}/\mathbb{Q}{})\mid\sigma\Phi
=\Phi\}$.

Let $(E,\Phi)$ be a CM-pair, and let $k$ be a subfield of $\mathbb{C}{}$.
There exists a finitely generated $E\otimes_{\mathbb{Q}{}}k$-module $V$ such
that\footnote{To give an $E\otimes_{\mathbb{Q}}k$-module structure on a
$\mathbb{Q}$-vector space $V$ is the same as to give commuting actions of $E$
and $k$. An element $a$ of $E$ defines a $k$-linear map $v\mapsto av\colon
V\rightarrow V$ whose trace we denote by $\Tr_{k}(a|V)$.}%
\begin{equation}
\Tr_{k}(a|V)=\Tr_{\Phi}(a)\quad\text{for all }a\in E \label{e01}%
\end{equation}
if and only if $k\supset E^{\ast}$, in which case $V$ is uniquely determined
up to an $E\otimes_{\mathbb{Q}{}}k$-isomorphism. For example, if $k$ contains
all conjugates of $E$, then $V$ must be $\bigoplus\nolimits_{\varphi\in\Phi
}k_{\varphi}$ where $k_{\varphi}$ is a one-dimensional $k$-space on which $E$
acts through $\varphi$.

Now assume that $k$ has finite degree over $\mathbb{Q}{}$, and let $V_{\Phi}$
be an $E\otimes_{\mathbb{Q}{}}k$-module satisfying (\ref{e01}). An element $a$
of $k$ defines an $E$-linear map $v\mapsto av\colon V\rightarrow V$ whose
determinant we denote by $\det_{E}(a|V_{\Phi})$. If $a\in k^{\times}$, then
$\det_{E}(a|V_{\Phi})\in E^{\times}$, and so in this way we get a homomorphism%
\[
N_{k,\Phi}\colon k^{\times}\rightarrow E^{\times}.
\]
More generally, for any $\mathbb{Q}{}$-algebra $R$ and $a\in(k\otimes
_{\mathbb{Q}{}}R)^{\times}$, we obtain an element
\[
\det\nolimits_{E\otimes_{\mathbb{Q}{}}R}(a|V_{\Phi}\otimes_{\mathbb{Q}{}}%
R)\in(E\otimes_{\mathbb{Q}{}}R)^{\times},
\]
and hence a homomorphism%
\[
N_{k,\Phi}(R)\colon(k\otimes_{\mathbb{Q}{}}R)^{\times}\rightarrow
(E\otimes_{\mathbb{Q}{}}R)^{\times}%
\]
natural in $R$ and independent of the choice of $V_{\Phi}$. It is called the
\emph{reflex norm}. When $k=E^{\ast}$, we drop it from the notation. The
following formulas are easy to check (\cite{milneCM}, \S 1):%

\begin{equation}
N_{k,\Phi}=N_{\Phi}\circ\Nm_{k/E^{\ast}}, \label{e53}%
\end{equation}
($k\subset\mathbb{C}{}$ is a finite extension of $E^{\ast}$);%
\begin{equation}
N_{\Phi}(a)\cdot\iota_{E}N_{\Phi}(a)=\Nm_{k\otimes_{\mathbb{Q}{}}%
R/R}(a)\text{, all }a\in(k\otimes_{\mathbb{Q}{}}R)^{\times}, \label{e76}%
\end{equation}
($R$ is a $\mathbb{Q}{}$-algebra);%
\begin{equation}
N_{k,\Phi}(a{}{})=\prod\nolimits_{\varphi\in\Phi}\varphi^{-1}(\Nm_{k/\varphi
E}a{}{}),\quad a\in k^{\times}, \label{e55}%
\end{equation}
($k\subset\mathbb{\mathbb{C}{}}$ is a finite extension of $E^{\ast}$
containing all conjugates of $E$).

From (\ref{e55}), we see that $N_{k,\Phi}$ maps units in $\mathcal{O}{}_{k}$
to units in $\mathcal{O}{}_{E}$ when $k$ contains all conjugates of $E$. Now
(\ref{e53}) shows that this remains true without the condition on $k$.
Therefore, $N_{k,\Phi}$ is well-defined on principal ideals, and one sees
easily that it has a unique extension to all fractional ideals: if
$\mathfrak{a}{}^{h}=(a)$, then $N_{k,\Phi}(\mathfrak{a}{})=N_{k,\Phi}%
(a)^{1/h}$. The formulas (\ref{e53},\ref{e76},\ref{e55}) hold for ideals. If
$\mathfrak{a}{}$ is a fractional ideal of $E^{\ast}$ and $k$ is a number field
containing all conjugates of $E$, then (\ref{e55}) applied to the extension
$\mathfrak{a}{}^{\prime}$ of $\mathfrak{a}{}$ to a fractional ideal of $k$
gives
\begin{equation}
N_{\Phi}(\mathfrak{a}{})^{[k\colon E^{\ast}]}=\prod\nolimits_{\varphi\in\Phi
}\varphi^{-1}(\Nm_{k/\varphi E}\mathfrak{\mathfrak{a}{}}{}^{\prime}).
\label{e54}%
\end{equation}

\subsection{Riemann pairs; Riemann forms}

A \emph{Riemann pair} $(\Lambda,J)$ is a free $\mathbb{Z}{}$-module $\Lambda$
of finite rank together with a complex structure $J$ on $\mathbb{R}%
\otimes\Lambda$ (i.e., $J$ is an $\mathbb{R}{}$-linear endomorphism of
$\Lambda$ with square $-1$). A \emph{rational Riemann form} for a Riemann pair
is an alternating $\mathbb{Q}$-bilinear form $\psi\colon\Lambda_{\mathbb{Q}{}%
}\times\Lambda_{\mathbb{Q}{}}\rightarrow\mathbb{Q}$ such that%
\[
(x,y)\mapsto\psi_{\mathbb{R}{}}(x,Jy)\colon\Lambda_{\mathbb{R}{}}\times
\Lambda_{\mathbb{R}{}}\rightarrow\mathbb{R}{}%
\]
is symmetric and positive definite.

Let $(E,\Phi)$ be a CM-pair, and let $\Lambda$ be a lattice in $E$. Then
$\Phi$ defines an isomorphism
\[
e\otimes r\mapsto(\varphi(e)r)_{\varphi\in\Phi}\colon E\otimes_{\mathbb{Q}{}%
}\mathbb{R}{}\rightarrow\mathbb{C}{}^{\Phi}\text{,}%
\]
and so%
\[
\Lambda\otimes_{\mathbb{Z}{}}\mathbb{R}{}\simeq\Lambda\otimes_{\mathbb{Z}{}%
}\mathbb{Q}{}\otimes_{\mathbb{Q}{}}\mathbb{R}{}\simeq E\otimes_{\mathbb{Q}{}%
}\mathbb{R}{}\simeq\mathbb{C}^{\Phi},
\]
from which $\Lambda\otimes_{\mathbb{Z}{}}\mathbb{R}{}$ acquires a complex
structure $J_{\Phi}$. An $\alpha\in E^{\times}$ defines a $\mathbb{Q}{}%
$-bilinear form%
\[
(x,y)\mapsto\Tr_{E/\mathbb{Q}{}}(\alpha x\cdot\iota_{E}y)\colon E\times
E\rightarrow\mathbb{Q}{},
\]
which is a rational Riemann form if and only if
\begin{equation}
\iota_{E}\alpha=-\alpha\text{ and }\Im(\varphi(\alpha))>0\text{ for all
}\varphi\in\Phi; \label{e03}%
\end{equation}
every rational Riemann form is of this form for a unique $\alpha$.

Let $F$ be the product of the largest totally real subfields of the factors of
$E$. Then $E=F[\alpha]$ with $\alpha^{2}\in F$, which implies that $\iota
_{E}\alpha=-\alpha$. The weak approximation theorem shows that $\alpha$ can be
chosen so that $\Im(\varphi\alpha)>0$ for all $\varphi\in\Phi$. Thus, there
certainly exist $\alpha$s satisfying (\ref{e03}), and so $(\Lambda,J_{\Phi})$
admits a Riemann form.

Let $\alpha$ be one element of $E^{\times}$ satisfying (\ref{e03}). Then the
other such elements are exactly those of the form $a\alpha$ with $a$ a totally
positive element of $F^{\times}$. In other words, if $\psi$ is one rational
Riemann form, then the other rational Riemann forms are exactly those of the
form $a\psi$ with $a$ a totally positive element of $F^{\times}$%
.\label{riemann}

\subsection{Abelian varieties with complex multiplication}

Let $A$ be an abelian variety over a field $k$, and let $E$ be an \'{e}tale
$\mathbb{Q}{}$-subalgebra of $\End^{0}(A)\overset{\text{{\tiny def}}}%
{=}\End(A)\otimes\mathbb{Q}{}$. If $k$ can be embedded in $\mathbb{C}{}$, then
$\End^{0}(A)$ acts faithfully on $H_{1}(A(\mathbb{C}{}),\mathbb{Q}{})$, which
has dimension $2\dim A$, and so%
\begin{equation}
\lbrack E\colon\mathbb{Q}{}]\leq2\dim A. \label{e35}%
\end{equation}
In general, for $\ell\neq\mathrm{char\,}k$, $\End^{0}(A)\otimes_{\mathbb{Q}{}%
}\mathbb{Q}_{\ell}$ acts faithfully on $V_{\ell}A$, which again implies
(\ref{e35}). When equality holds we say that $A$ has \emph{complex
multiplication by} $E$ \emph{over} $k$. More generally, we say that $(A,i)$ is
an \emph{abelian variety with complex multiplication by} $E$ \emph{over} $k$
if $i$ is an injective homomorphism from an \'{e}tale $\mathbb{Q}{}$-algebra
$E$ of degree $2\dim A$ into $\End^{0}(A)$ (recall that this requires that
$i(1)$ acts as $\id_{A}$; see Notations).

\subsubsection{Classification up to isogeny}

\begin{plain}
\label{a30p}Let $A$ be an abelian variety with complex multiplication, so that
$\End^{0}(A)$ contains a CM-algebra $E$ for which $H_{1}(A,\mathbb{Q}{})$ is
free $E$-module of rank $1$, and let $\Phi$ be the set of homomorphisms
$E\rightarrow\mathbb{C}{}$ occurring in the representation of $E$ on
$\mathrm{Tgt}_{0}(A)$, i.e., $\mathrm{Tgt}_{0}(A)\simeq\bigoplus
\nolimits_{\varphi\in\Phi}\mathbb{C}{}_{\varphi}$ where $\mathbb{C}{}%
_{\varphi}$ is a one-dimensional $\mathbb{C}{}$-vector space on which $a\in E$
acts as $\varphi(a)$. Then, because%
\begin{equation}
H_{1}(A,\mathbb{R}{})\simeq\mathrm{Tgt}_{0}(A)\oplus\overline{\mathrm{Tgt}%
_{0}(A)} \label{e18}%
\end{equation}
$\Phi_{A}$ is a CM-type on $E$, and we say that $A$ together with the
injective homomorphism $i\colon E\rightarrow\End^{0}(A)$ is of \emph{type}
$(E,\Phi)$.

Let $e$ be a basis vector for $H_{1}(A,\mathbb{Q}{})$ as an $E$-module, and
let $\mathfrak{a}{}$ be the lattice in $E$ such that $\mathfrak{a}{}%
e=H_{1}(A,\mathbb{Z}{})$. Under the isomorphism (cf. (\ref{e18}))%
\[
H_{1}(A,\mathbb{R}{})\simeq\bigoplus\nolimits_{\varphi\in\Phi}\mathbb{C}%
{}_{\varphi}\oplus\bigoplus\nolimits_{\varphi\in\iota\Phi}\mathbb{C}%
{}_{\varphi},
\]%
\[
e\otimes1\longleftrightarrow(\ldots,e_{\varphi},\ldots;\ldots,e_{\iota
\circ\varphi},\ldots)
\]
where each $e_{\varphi}$ is a $\mathbb{C}{}$-basis for $\mathbb{C}{}_{\varphi
}$. The $e_{\varphi}$ determine an isomorphism%
\[
\mathrm{Tgt}_{0}(A)\simeq\bigoplus\nolimits_{\varphi\in\Phi}\mathbb{C}%
{}_{\varphi}\overset{e_{\varphi}}{\simeq}\mathbb{C}{}^{\Phi},
\]
and hence a commutative square of isomorphisms in which the top arrow is the
canonical uniformization:%
\begin{equation}
\begin{CD} \Tgt_{0}(A)/\Lambda @>>> A\\ @VVV@VVV\\ \mathbb{C}^{\Phi}/\Phi(\mathfrak{a}) @= A_{\Phi}.\end{CD} \label{e19}%
\end{equation}

\end{plain}

\begin{proposition}
\label{a31a}The map $(A,i)\mapsto(E,\Phi)$ gives a bijection from the set of
isogeny classes of pairs $(A,i)$ to the set of isomorphism classes of CM-pairs.
\end{proposition}

\subsubsection{Classification up to isomorphism}

Let $(A,i)$ be of CM-type $(E,\Phi)$. Let $e$ be an $E$-basis element of
$H_{1}(A,\mathbb{Q}{})$, and set $H_{1}(A,\mathbb{Z}{})=\mathfrak{a}{}e$ with
$\mathfrak{a}{}$ a lattice in $E$. We saw in (\ref{a30p}) that $e$ determines
an isomorphism%
\[
\theta\colon(A_{\Phi},i_{\Phi})\rightarrow(A,i),\quad A_{\Phi}\overset
{\text{{\tiny def}}}{=}\mathbb{C}{}^{\Phi}/\Phi(\mathfrak{a}{}).
\]
Conversely, every isomorphism $\mathbb{C}{}^{\Phi}/\Phi(\mathfrak{a}%
{})\rightarrow A$ commuting with the actions of $E$ arises in this way from an
$E$-basis element of $H_{1}(A,\mathbb{Q}{})$, because%
\[
E\simeq H_{1}(A_{\Phi},\mathbb{Q}{})\overset{\theta}{\simeq}H_{1}%
(A,\mathbb{Q}{}).
\]
If $e$ is replaced by $ae$, $a\in E^{\times}$, then $\theta$ is replaced by
$\theta\circ a^{-1}$.

We use this observation to classify triples $(A,i,\mathcal{\psi})$ where $A$
is an abelian variety, $i\colon E\rightarrow\End^{0}(A)$ is a homomorphism
making $H_{1}(A,\mathbb{Q}{})$ into a free module of rank $1$ over the
CM-algebra $E$, and $\psi$ is a rational Riemann form whose Rosati involution
stabilizes $i(E)$ and induces $\iota_{E}$ on it.

Let $\theta\colon\mathbb{C}{}^{\Phi}/\Phi(\mathfrak{a}{})\rightarrow A$ be the
isomorphism defined by some basis element $e$ of $H_{1}(A,\mathbb{Q}{})$. Then
(see p\pageref{riemann}), there exists a unique element $t\in E^{\times}$ such
that $\psi(xe,ye)=\Tr_{E/\mathbb{Q}{}}(tx\bar{y})$. The triple
$(A,i,\mathcal{\psi})$ is said to be of \emph{type} $(E,\Phi;\mathfrak{a}%
{},t)$ relative to $\theta$ (cf. \cite{shimura1971}, Section 5.5 B).

\begin{proposition}
\label{a34}The type $(E,\Phi;\mathfrak{a}{},t)$ determines $(A,i,\psi)$ up to
isomorphism. Conversely, $(A,i,\psi)$ determines the type up to a change of
the following form: if $\theta$ is replaced by $\theta\circ a^{-1}$, $a\in
E^{\times}$, then the type becomes $(E,\Phi;a\mathfrak{a}{},t/a\bar{a})$. The
quadruples $(E,\Phi;\mathfrak{a}{},t)$ that arise as the type of some triple
are exactly those in which $(E,\Phi)$ is a CM-pair, $\mathfrak{a}{}$ is a
lattice in $E$, and $t$ is an element of $E^{\times}$ such that $\iota
_{E}t=-t$ and $\Im(\varphi(t))>0$ for all $\varphi\in\Phi$.
\end{proposition}

\begin{proof}
Routine verification.
\end{proof}

\subsubsection{Commutants}

Let $A$ have complex multiplication by $E$ over $k$, and let
\[
R=E\cap\End(A).
\]
Then $R$ is an \emph{order} in $E$, i.e., it is simultaneously a subring and a
lattice in $E$.

Let $g=\dim A$, and let $\ell$ be a prime not equal to $\mathrm{char\,}k$.
Then $T_{\ell}A$ is a $\mathbb{Z}{}_{\ell}$-module of rank $2g$ and $V_{\ell
}A$ is a $\mathbb{Q}{}_{\ell}$-vector space of dimension $2g$. The action of
$R$ on $T_{\ell}A$ extends to actions of $R_{\ell}\overset{\text{{\tiny def}}%
}{=}R\otimes_{\mathbb{Z}{}}\mathbb{Z}{}_{\ell}$ on $T_{\ell}A$ and of
$E_{\ell}\overset{\text{{\tiny def}}}{=}\mathbb{Q}{}_{\ell}\otimes
_{\mathbb{Q}{}}E$ on $V_{\ell}A$.

\begin{proposition}
\label{b03a}(a) The $E_{\ell}$-module $V_{\ell}A$ is free of rank $1$.

(b) We have%
\[
R_{\ell}=E_{\ell}\cap\End(T_{\ell}A).
\]

\end{proposition}

\begin{proof}
(a) We have already noted that $E_{\ell}$ acts faithfully on $V_{\ell}A$, and
this implies that $V_{\ell}A$ is free of rank $1$.

(b) Let $\alpha$ be an element of $E_{\ell}$ such that $\alpha(T_{\ell
}A)\subset T_{\ell}A$. For some $m$, $\ell^{m}\alpha\in R_{\ell}$, and if
$\beta\in R$ is chosen to be very close $\ell$-adically to $\ell^{m}\alpha$,
then $\beta T_{\ell}A\subset\ell^{m}T_{\ell}A$, which means that $\beta$
vanishes on $A_{\ell^{m}}$. Hence $\beta=\ell^{m}\alpha_{0}$ for some
$\alpha_{0}\in\End(A)\cap E=R$. Now $\alpha$ and $\alpha_{0}$ are close in
$E_{\ell}$; in particular, we may suppose $\alpha-\alpha_{0}\in R_{\ell}$, and
so $\alpha\in R_{\ell}$.
\end{proof}

\begin{corollary}
\label{b03b}The commutants of $R$ in $\End_{\mathbb{Q}{}_{\ell}}(V_{\ell}A)$,
$\End_{\mathbb{Z}{}_{\ell}}(T_{\ell}A)$, $\End^{0}(A)$, and $\End(A)$ are,
respectively, $E_{\ell}$, $R_{\ell}$, $F$, and $R$.
\end{corollary}

\begin{proof}
Any endomorphism of $V_{\ell}A$ commuting with $R$ commutes with $E_{\ell}$,
and therefore lies in $E_{\ell}$, because of (\ref{b03a}a).

Any endomorphism of $T_{\ell}A$ commuting with $R$ extends to an endomorphism
of $V_{\ell}A$ preserving $T_{\ell}A$ and commuting with $R$, and so lies in
$E_{\ell}\cap\End(T_{\ell}A)=R_{\ell}$.

Let $C$ be the commutant of $E$ in $\End^{0}(A)$. Then $E$ is a subalgebra of
$C$, so $[E\colon\mathbb{Q}{}]\leq\lbrack C\colon\mathbb{Q}{}]$, and
$C\otimes_{\mathbb{Q}{}}\mathbb{Q}_{\ell}$ is contained in the commutant
$E_{\ell}$ of $E$ in $\End(V_{\ell}A)$, so $[E\colon\mathbb{Q}{}]\geq\lbrack
C\colon\mathbb{Q}{}]$. Thus $E=C$.

Finally, the commutant $R$ in $\End(A)$ contains $R$ and is contained in
$C\cap\End(A)=E\cap\End(A)=R$.
\end{proof}

\begin{corollary}
\label{b03d}Let $(A,i)$ have complex multiplication by $E$, and let
$R=i^{-1}(\End(A))$. Then any endomorphism of $A$ commuting with $i(a)$ for
all $a\in R$ is of the form $i(b)$ for some $b\in R$.
\end{corollary}

\begin{proof}
Apply the preceding corollary to $i(E)\subset\End^{0}(A).$
\end{proof}

\begin{remark}
\label{b03c}If $\ell$ does not divide $(\mathcal{O}{}_{E}\colon R)$, then
$R_{\ell}$ is a product of discrete valuation rings, and $T_{\ell}A$ is a free
$R_{\ell}$-module of rank $1$, but in general this need not be true
(\cite{serreT1968}, p502). Similarly, $T_{m}A\overset{\text{{\tiny def}}}%
{=}\prod_{\ell|m}T_{\ell}A$ is a free $R{}_{m}\overset{\text{{\tiny def}}}%
{=}\prod_{\ell|m}R_{\ell}$-module of rank $1$ if $m$ is relatively prime to
$(\mathcal{O}{}_{E}\colon R)$.
\end{remark}

Let $(A,i)$ be an abelian variety with complex multiplication by a CM-algebra
$E$ over a field $k$ of characteristic zero. If $k$ contains all conjugates of
$E$, then $\mathrm{Tgt}_{0}(A)\simeq\prod\nolimits_{\varphi\in\Phi}k_{\varphi
}$ as an $E\otimes_{\mathbb{Q}{}}k$-module where $\Phi$ is a set of
$\mathbb{Q}{}$-algebra homomorphisms $E\hookrightarrow k$ and $k_{\varphi}$ is
a one-dimensional $k$-vector space on which $a\in E$ acts as $\varphi(a)$. For
any complex conjugation\footnote{A \emph{complex conjugation} on a field $k$
is the involution induced by complex conjugation on $\mathbb{C}{}$ through
some embedding of $k$ into $\mathbb{C}{}$.} $\iota$ on $k$,%
\[
\Phi\sqcup\iota\Phi=\Hom(E,k)\text{.}%
\]
A subset $\Phi$ of $\Hom(E,k)$ with this property will be called a CM-type on
$E$ with values in $k$. If $k\subset\mathbb{C}{}$, then it can also be
regarded as a CM-type on $E$ with values in $\mathbb{C}{}$.

\subsection{Extension of the base field}

Let $k$ be an algebraically closed subfield of $\mathbb{C}{}$. For abelian
varieties $A,B$ over $k$, $\Hom(A,B)\simeq\Hom(A_{\mathbb{C}{}},B_{\mathbb{C}%
{}})$, i.e., the functor from abelian varieties over $k$ to abelian varieties
over $\mathbb{C}{}$ is fully faithful. It is even essentially surjective
(hence an equivalence) on abelian varieties with complex multiplication. See,
for example, \cite{milneCM}, Proposition 7.8.\label{extension}

\subsection{Good reduction}

Let $R$ be a discrete valuation ring with field of fractions $K$ and residue
field $k$. An abelian variety $A$ over $K$ is said to have \emph{good
reduction} if it is the generic fibre of an abelian scheme $\mathcal{A}{}$
over $R$. Then the special fibre $A_{0}$ of $\mathcal{A}{}$ is an abelian
variety, and $\Tgt_{0}(\mathcal{A}{})$ is a free $R$-module such that%
\begin{align*}
\Tgt_{0}(\mathcal{A}{})\otimes_{R}K &  \simeq\Tgt_{0}(A)\\
\Tgt_{0}(\mathcal{A}{})\otimes_{R}k &  \simeq\Tgt_{0}(A_{0}).
\end{align*}
The map%
\[
\End(\mathcal{A}{})\rightarrow\End(A)
\]
is an isomorphism, and there is a reduction map%
\begin{equation}
\End(A)\simeq\End(\mathcal{A}{})\rightarrow\End(A_{0})\text{.}\label{e3}%
\end{equation}
This is an injective homomorphism. See, for example, \cite{milneCM}, II, \S 6.

It is a fairly immediate consequence of N\'{e}ron's theorem on the existence
of minimal models that an abelian variety with complex multiplication over a
number field $k$ acquires good reduction at all finite primes after finite
extension of $k$ (\cite{serreT1968}, Theorem 6; \cite{milneCM},
7.12).\footnote{N\'{e}ron's theorem was, of course, not available to Shimura
and Taniyama, who proved their results \textquotedblleft for almost all
$\mathfrak{p}{}$\textquotedblright. N\'{e}ron's theorem allowed later
mathematicians to claim to have sharpened the results of Shimura and Taniyama
without actually having done anything.}

\subsection{The degrees of isogenies}

An isogeny $\alpha\colon A\rightarrow B$ defines a homomorphism $\alpha^{\ast
}\colon k(B)\rightarrow k(A)$ of the fields of rational functions, and the
\emph{degree} of $\alpha$ is defined to be $[k(A)\colon\alpha^{\ast}k(B)]$.

\begin{proposition}
\label{b17}Let $A$ be an abelian variety with complex multiplication by $E$,
and let $R=\End(A)\cap E$. An element $\alpha$ of $R$ that is not a
zero-divisor is an isogeny of degree $(R\colon\alpha R)$.
\end{proposition}

\begin{proof}
If $\alpha$ is not a zero-divisor, then it is invertible in $E\simeq
R\otimes_{\mathbb{Z}{}}\mathbb{Q}{}$, and so it is an isogeny. Let $d$ be its
degree, and choose a prime $\ell$ not dividing $d\cdot\mathrm{char(}k)$. Then
$d$ is the determinant of $\alpha$ acting on $V_{\ell}A$ (e.g.,
\cite{milne1986ab}, 12.9). As $V_{\ell}A$ is free of rank $1$ over $E_{\ell
}\overset{\text{{\tiny def}}}{=}E\otimes_{\mathbb{Q}{}}\mathbb{Q}{}_{\ell}$,
this determinant is equal to $\Nm_{E_{\ell}/\mathbb{Q}{}_{\ell}}(\alpha)$,
which equals $\Nm_{E/\mathbb{Q}{}}(\alpha)$. But $R$ is a lattice in $E$, and
so this norm equals $(R\colon\alpha R)$.\footnote{In more detail: let
$e_{1},\ldots,e_{n}$ be a basis for $R$ as a $\mathbb{Z}{}$-module, and let
$\alpha e_{j}=\sum\nolimits_{i}a_{ij}e_{i}$. For some $\varepsilon\in V_{\ell
}A$, $e_{1}\varepsilon,\ldots,e_{n}\varepsilon$ is a $\mathbb{Q}{}_{\ell}%
$-basis for $V_{\ell}A$. As $\alpha e_{j}\varepsilon=\sum\nolimits_{i}%
a_{ij}e_{i}\varepsilon$, we have that $d=\det(a_{ij})$. But $\left\vert
\det(a_{ij})\right\vert =(R:\alpha R)$ (standard result, which is obvious, for
example, if $\alpha$ is diagonal).}
\end{proof}

\begin{proposition}
[\cite{shimuraT1961}, I 2.8, Thm 1]\label{b18}Let $k$ be an algebraically
closed field of characteristic $p>0$, and let $\alpha\colon A\rightarrow B$ be
an isogeny of abelian varieties over $k$. Assume that $\alpha^{\ast
}(k(B))\supset k(A)^{q}$ for some power $q=p^{m}$ of $p$, and let $d$ be the
dimension of the kernel of $\mathrm{Tgt}_{0}(\alpha)\colon\Tgt_{0}%
(A)\rightarrow\Tgt_{0}(B)$; then%
\[
\deg(\alpha)\leq q^{d}\text{.}%
\]

\end{proposition}

We offer two proofs, according to the taste and knowledge of the reader.

\subsubsection{Proof of (\ref{b18}) in terms of varieties and differentials}

\begin{lemma}
\label{b18a}Let $L/K$ be a finitely generated extension of fields of
characterstic $p>0$ such that $K\supset L^{q}$ for some power $q\ $of $p$.
Then%
\[
\lbrack L\colon K]\leq q^{\dim\Omega_{L/K}^{1}}.
\]

\end{lemma}

\begin{proof}
We use that $\Hom_{K\text{-linear}}(\Omega_{L/K}^{1},K)$ is isomorphic to the
space of $K$-derivations $L\rightarrow K$. Let $x_{1},\ldots,x_{n}$ be a
minimal set of generators for $L$ over $K$. Because $x_{i}^{q}\in K$,
$[L\colon K]<q^{n}$, and it remains to prove $\dim\Omega_{L/K}^{1}\geq n$. For
each $i$, $L$ is a purely inseparable extension of $K(x_{1},\ldots
,x_{i-1},x_{i+1},\ldots,x_{n})$ because $L\supset K\supset L^{q}$. There
therefore exists a $K$-derivation of $D_{i}$ of $L$ such that $D_{i}%
(x_{i})\neq0$ but $D_{i}(x_{j})=0$ for $j\neq i$, namely, $\frac{\partial
}{\partial x_{i}}$. The $D_{i}$ are linearly independent, from which the
conclusion follows.
\end{proof}

\begin{pf}
[of \ref{b18}]In the lemma, take $L=k(A)$ and $K=\alpha^{\ast}(k(B))$. Then
$\deg(\alpha)=[L\colon K]$ and $\dim\Omega_{L/K}^{1}=\dim\Ker(\Tgt_{0}%
(\alpha))$, and so the proposition follows.
\end{pf}

\subsubsection{Proof of (\ref{b18}) in terms of finite group schemes}

The \emph{order} of a finite group scheme $N=\Spec R$ over a field $k$ is
$\dim_{k}R$.

\begin{lemma}
\label{b18b}The kernel of an isogeny of abelian varieties is a finite group
scheme of order equal to the degree of the isogeny.
\end{lemma}

\begin{proof}
Let $\alpha\colon A\rightarrow B$ be an isogeny. Then (e.g.,
\cite{milne1986ab}, 8.1) $\alpha_{\ast}\mathcal{O}{}_{A}$ is a locally free
$\mathcal{O}{}_{B}$-module, of rank $r$ say. The fibre of $\alpha_{\ast
}\mathcal{O}{}_{A}$ at $0_{B}$ is the affine ring of $\Ker(\alpha)$, which
therefore is finite of order $r$. The fibre of $\alpha_{\ast}\mathcal{O}{}%
_{A}$ at the generic point of $B$ is $k(A)$, and so $r=[k(A)\colon\alpha
^{\ast}k(B)]=\deg(\alpha)$.
\end{proof}

\begin{pf}
[of \ref{b18}]The condition on $\alpha$ implies that $\Ker(\alpha)$ is
connected, and therefore its affine ring is of the form $k[T_{1},\ldots
,T_{s}]/(T_{1}^{p^{r_{1}}},\ldots,T_{s}^{p^{r_{s}}})$ for some family
$(r_{i})_{1\leq i\leq s}$ of integers $r_{i}\geq1$ (\cite{waterhouse1979},
14.4). Let $q=p^{m}$. Then each $r_{i}\leq m$ because $\alpha^{\ast
}(k(B))\supset k(A)^{q}$, and
\[
s=\dim_{k}\mathrm{Tgt}_{0}(\Ker(\alpha))=\dim_{k}\Ker(\mathrm{Tgt}_{0}%
(\alpha))=d.
\]
Therefore,%
\[
\deg(\alpha)=\prod\nolimits_{i=1}^{s}p^{r_{i}}\leq p^{ms}=q^{d}.
\]

\end{pf}

\subsection{$\mathfrak{a}$-multiplications: first approach}

Let $A$ be an abelian variety with complex multiplication by $E$ over a field
$k$, and let $R=E\cap\End(A)$. An element of $R$ is an isogeny if and only if
it is not a zero-divisor,\footnote{Recall that $E$ is an \'{e}tale
$\mathbb{Q}{}$-subalgebra of $\End^{0}(A)$, i.e., a product of fields, say
$E=\prod E_{i}$. Obviously $E=R\otimes_{\mathbb{Z}{}}\mathbb{Q}{}$, and
$R\subset E$. An element $\alpha=(\alpha_{i})$ of $R$ is not zero-divisor if
and only if each component $\alpha_{i}$ of $\alpha$ is nonzero, or,
equivalently, $\alpha$ is an invertible element of $E$.} and an ideal
$\mathfrak{a}{}$ in $R$ contains an isogeny if and only if it is a lattice in
$E$ --- we call ideals with this property \emph{lattice ideals}. We wish to
attach to each lattice ideal $\mathfrak{a}{}$ in $R$ an isogeny $\lambda
^{\mathfrak{a}{}}\colon A\rightarrow A^{\mathfrak{a}{}}$ with certain
properties. The shortest definition is to take $A^{\mathfrak{a}{}}$ to be the
quotient of $A$ by the finite group scheme%
\[
\Ker(\mathfrak{a})=%
{\textstyle\bigcap_{a\in\mathfrak{a}{}}}
\Ker(a).
\]
However, the formation of quotients by finite group schemes in characteristic
$p$ is subtle (\cite{mumford1970}, p109-123)\footnote{Compare the proof of
(\ref{b32}) with that of \cite{mumford1970}, III, Theorem 1, p111.}, and was
certainly not available to Shimura and Taniyama. In this subsection, we give
an elementary construction.

\begin{definition}
\label{b30} Let $A$ be an abelian variety with complex multiplication by $E$
over a field $k$, and let $\mathfrak{a}{}$ be a lattice ideal in $R$. A
surjective homomorphism $\lambda^{\mathfrak{a}{}}\colon A\rightarrow
A^{\mathfrak{a}{}}$ is an $\mathfrak{a}$-\emph{multiplication} if every
homomorphism $a\colon A\rightarrow A$ with $a\in\mathfrak{a}{}$ factors
through $\lambda^{\mathfrak{a}{}}$, and $\lambda^{\mathfrak{a}{}}$ is
universal for this property, in the sense that, for every surjective
homomorphism $\lambda^{\prime}\colon A\rightarrow A^{\prime}$ with the same
property, there is a homomorphism $\alpha\colon A^{\prime}\rightarrow
A^{\mathfrak{a}{}}$, necessarily unique, such that $\alpha\circ\lambda
^{\prime}=\lambda^{\mathfrak{a}{}}$:%
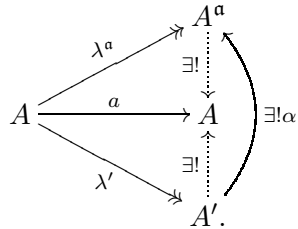
\[
\xymatrix{
&&A^{\mathfrak{a}}\ar@{.>}[d]_{\exists!}\\
A\ar@{->>}[rru]^{\lambda^{\mathfrak{a}}}\ar[rr]^{a}\ar@{->>}[rrd]_{\lambda^{\prime}}&&A\\
&&A^{\prime}\ar@{.>}[u]^{\exists!}\ar@/_1.5pc/[uu]_{\exists!\alpha}.\\
}
\]
An abelian variety $B$ for which there exists an $\mathfrak{a}{}%
$-multiplication $A\rightarrow B$ is called an $\mathfrak{a}{}$%
-\emph{transform} of $A$.
\end{definition}

\begin{example}
\label{b31}(a) If $\mathfrak{a}{}$ is principal, say, $\mathfrak{a}{}=(a)$,
then $a\colon A\rightarrow A$ is an $\mathfrak{a}$-multiplication (obvious
from the definition) --- this explains the name \textquotedblleft%
$\mathfrak{a}{}$-multiplication\textquotedblright. More generally, if
$\lambda\colon A\rightarrow A^{\prime}$ is an $\mathfrak{a}{}$-multiplication,
then%
\[
A\overset{a}{\longrightarrow}A\overset{\lambda}{\longrightarrow}A^{\prime}%
\]
is an $\mathfrak{a}a{}$-multiplication for any $a\in E$ such that
$\mathfrak{a}{}a\subset R$ (obvious from the construction in \ref{b32} below).

(b) Let $(E,\Phi)$ be a CM-pair, and let $A=\mathbb{C}{}^{\Phi}/\Phi(\Lambda)$
for some lattice $\Lambda$ in $E$. For any lattice ideal $\mathfrak{a}{}$ in
$R\overset{\text{{\tiny def}}}{=}\End(A)\cap E$,
\begin{align*}
\Ker(\mathfrak{a}{})  &  =\{z+\Phi(\Lambda)\mid az\in\Phi(\Lambda)\text{ all
}a\in\mathfrak{\mathfrak{a}{}}\}\\
&  =\Phi(\mathfrak{a}{}^{-1}\Lambda)/\Phi(\Lambda)
\end{align*}
where $\mathfrak{a}{}^{-1}=\{a\in E\mid a\mathfrak{a}{}\subset R\}$. The
quotient map $\mathbb{C}{}^{\Phi}/\Phi(\Lambda)\rightarrow\mathbb{C}{}^{\Phi
}/\Phi(\mathfrak{a}^{-1}{}\Lambda)$ is an $\mathfrak{a}$-multiplication.
\end{example}

\begin{remark}
\label{b31r}(a) The universal property shows that an $\mathfrak{a}{}%
$-multiplication, if it exists, is unique up to a unique isomorphism.

(b) Let $a\in\mathfrak{a}{}$ be an isogeny; because $a$ factors through
$\lambda^{\mathfrak{a}{}}$, the map $\lambda^{\mathfrak{a}{}}$ is an isogeny.

(c) The universal property, applied to $\lambda^{\mathfrak{a}{}}\circ a$ for
$a\in R$, shows that, $A^{\mathfrak{a}{}}$ has complex multiplication by $E$
over $k$, and $\lambda^{\mathfrak{a}{}}$ is an $E$-isogeny. Moreover,
$R\subset\End(A^{\mathfrak{a}{}})\cap E$, but the inclusion may be strict
unless $R=\mathcal{O}{}_{E}$.\footnote{Over $\mathbb{C}{}$, $A$ is
$E$-isogenous to an abelian variety with $\End(A)\cap E=\mathcal{O}{}_{E}$,
but every such isogeny is an $\mathfrak{a}{}$-multiplication for some
$\mathfrak{a}{}$ (see below).}

(d) If $\lambda\colon A\rightarrow B$ is an $\mathfrak{a}{}$-multiplication,
then so also is $\lambda_{k^{\prime}}\colon A_{k^{\prime}}\rightarrow
B_{k^{\prime}}$ for any $k^{\prime}\supset k$. This follows from the
construction in (\ref{b32}) below.
\end{remark}

\begin{proposition}
\label{b32}An $\mathfrak{a}{}$-multiplication exists for each lattice ideal
$\mathfrak{a}{}.$
\end{proposition}

\begin{proof}
Choose a set of generators $a_{1},...,a_{n}$ of $\mathfrak{a}$, and define
$A^{\mathfrak{a}{}}$ to be the image of
\begin{equation}
x\mapsto(a_{1}x,\ldots)\colon A\rightarrow A^{n}. \label{e39}%
\end{equation}
For any $a=\sum\nolimits_{i}r_{i}a_{i}\in\mathfrak{a}{}$, the diagram%
\[
\xymatrixcolsep{4pc}\xymatrix{
A\ar@/_1pc/[rr]_a\ar[r]^{\left(\begin{smallmatrix}a_1\\\vdots\\a_n\end{smallmatrix}\right)}
&A^n\ar[r]^{\left(\begin{smallmatrix}r_1,&\cdots,&r_n\end{smallmatrix}\right)}&A}
\]
shows that $a\colon A\rightarrow A$ factors through $\lambda^{\mathfrak{a}{}}$.

Let $\lambda^{\prime}\colon A\rightarrow A^{\prime}$ be a quotient map such
that each $a_{i}$ factors through $\lambda^{\prime}$, say, $\alpha_{i}%
\circ\lambda^{\prime}=a_{i}$. Then the composite of%
\begin{equation}
\begin{CD} A @>{\lambda^{\prime}}>> A^{\prime} @>{\alpha=\left(\begin{smallmatrix}\alpha_1\\\vdots\\\alpha_n\end{smallmatrix}\right)}>>A^n \end{CD} \label{e41}%
\end{equation}
is $x\mapsto(a_{1}x,\ldots)\colon A\rightarrow A^{n}$, which shows that
$\alpha\circ\lambda^{\prime}=\lambda^{\mathfrak{a}{}}$.
\end{proof}

\begin{remark}
\label{b32r}A surjective homomorphism $\lambda\colon A\rightarrow B$ is an
$\mathfrak{a}{}$-multiplication if and only if every homomorphism $a\colon
A\rightarrow A$ defined by an element of $\mathfrak{a}{}$ factors through
$\lambda$ and one (hence every) family $(a_{i})_{1\leq i\leq n}$ of generators
for $\mathfrak{a}{}$ defines an isomorphism of $B$ onto the image of $A$ in
$A^{n}$. Alternatively, a surjective homomorphism $\lambda\colon A\rightarrow
B$ is an $\mathfrak{a}{}$-multiplication if it maps $k(B)$ isomorphically onto
the composite of the fields $a^{\ast}k(A)$ for $a\in\mathfrak{a}{}$ --- this
is the original definition (\cite{shimuraT1961}, 7.1).
\end{remark}

\begin{proposition}
\label{b37}Let $A$ be an abelian variety with complex multiplication by $E$
over $k$, and assume that $E\cap\End(A)=\mathcal{O}{}_{E}$. Let $\lambda\colon
A\rightarrow B$ and $\lambda^{\prime}\colon A\rightarrow B^{\prime}$ be
$\mathfrak{a}$ and $\mathfrak{a^{\prime}}$-multiplications respectively. There
exists an $E$-isogeny $\alpha\colon B\rightarrow B^{\prime}$ such that
$\alpha\circ\lambda=\lambda^{\prime}$ if only if
$\mathfrak{\mathfrak{\mathfrak{a}{}}}\supset\mathfrak{a}{}^{\prime
}\mathfrak{.}$
\end{proposition}

\begin{proof}
If $\mathfrak{a}{}\supset\mathfrak{a}{}^{\prime}$, then $a\colon A\rightarrow
A$ factors through $\lambda$ when $a\in\mathfrak{a}{}^{\prime}$, and so
$\alpha$ exists by the universality of $\lambda^{\prime}$. For the converse,
note that there are natural quotient maps $A^{\mathfrak{a}{}+\mathfrak{a}%
{}^{\prime}}\rightarrow A^{\mathfrak{a}{}},A^{\mathfrak{a}{}^{\prime}}$. If
there exists an $E$-isogeny $\alpha$ such that $\alpha\circ\lambda
^{\mathfrak{a}{}}=\lambda^{\mathfrak{a}{}^{\prime}}$, then $A^{\mathfrak{a}%
{}+\mathfrak{a}{}^{\prime}}\rightarrow A^{\mathfrak{a}{}}$ is injective, which
implies that $\mathfrak{a}{}+\mathfrak{a}{}^{\prime}=\mathfrak{a}{}$ by
(\ref{b35}) below.
\end{proof}

\begin{corollary}
\label{b37c}Let $\lambda\colon A\rightarrow B$ and $\lambda^{\prime}\colon
A\rightarrow B^{\prime}$ be $\mathfrak{a}{}$ and $\mathfrak{a}{}^{\prime}%
$-multiplications; if there exists an $E$-isomorphism $\alpha\colon
B\rightarrow B^{\prime}$ such that $\alpha\circ\lambda=\lambda^{\prime}$, then
$\mathfrak{a}{}=\mathfrak{a}{}^{\prime}$.
\end{corollary}

\begin{proof}
The existence of $\alpha$ implies that $\mathfrak{a}{}\supset\mathfrak{a}%
{}^{\prime}$, and the existence of its inverse implies that $\mathfrak{a}%
{}^{\prime}\supset\mathfrak{a}{}$.
\end{proof}

\begin{corollary}
\label{b37d}Let $a\in\End(A)\cap E$. If $a\colon A\rightarrow A$ factors
through an $\mathfrak{a}{}$-multiplication, then $a\in\mathfrak{a}{}$.
\end{corollary}

\begin{proof}
The map $a\colon A\rightarrow A$ is an $(a)$-multiplication, and so if there
exists an $E$-isogeny $\alpha$ such that $\alpha\circ\lambda^{\mathfrak{a}{}%
}=a$, then $\mathfrak{a}{}\supset(a)$.
\end{proof}

\begin{remark}
\label{b33a}Let $\lambda\colon A\rightarrow B$ be an $\mathfrak{a}{}%
$-multiplication. Let $a_{1},\ldots,a_{n}$ be a basis for $\mathfrak{a}{}$,
and let $a_{i}=\alpha_{i}\circ\lambda$. In the diagram%
\[
\xymatrix{
A\ar[r]^\lambda\ar@/_1pc/[rr]_a&B\ar[r]^{\alpha}&A^n}\qquad\alpha={\left(
\begin{smallmatrix}
\alpha_{1}\\
\vdots\\
\alpha_{n}%
\end{smallmatrix}
\right)  }\quad a={\left(
\begin{smallmatrix}
a_{1}\\
\vdots\\
a_{n}%
\end{smallmatrix}
\right)  ,}%
\]
$\alpha$ maps $B$ isomorphically onto the image of $a$. For any prime $\ell$
different from the characteristic of $k$, we get a diagram%
\[
\xymatrix{
T_{\ell}A\ar[r]^{T_{\ell}\lambda}\ar@/_1pc/[rr]_{T_{\ell}a}&T_{\ell}B\ar[r]^{T_{\ell}\alpha}&T_{\ell}A^n}
\]
in which $T_{\ell}\alpha$ maps $T_{\ell}B$ isomorphically onto the image of
$T_{\ell}a$.
\end{remark}

\begin{proposition}
\label{b34}If $\lambda\colon A\rightarrow A^{\prime}$ is an $\mathfrak{a}%
$-multiplication, and $\lambda^{\prime}\colon A^{\prime}\rightarrow
A^{\prime\prime}$ is an $\mathfrak{a^{\prime}}$-multiplication, then
$\lambda^{\prime}\circ\lambda$ is an $\mathfrak{a^{\prime}a}$-multiplication.
\end{proposition}

\begin{proof}
Let $\mathfrak{a}=(a_{1},...,a_{m})$, and let $\mathfrak{a}^{\prime}%
=(a_{1}^{\prime},...,a_{m}^{\prime})$; then $\mathfrak{a}{}^{\prime
}\mathfrak{a}{}=(\ldots,a_{i}^{\prime}a_{j},\ldots)$, and one can show that
$A^{\prime\prime}$ is isomorphic to the image of $A$ under $x\mapsto
(\ldots,a_{i}^{\prime}a_{j}x,\ldots)$ (alternatively, use (\ref{b41c}) and
(\ref{e77})).
\end{proof}

\begin{proposition}
\label{b35}For any $\mathfrak{a}$-multiplication $\lambda$, $\deg
(\lambda)=(\mathcal{O}_{E}\colon\mathfrak{a)}$ provided $\mathfrak{a}{}$ is
invertible (locally free of rank $1$).
\end{proposition}

\begin{proof}
For simplicity, we assume that $\mathcal{O}{}_{E}=\End(A)\cap E$. According to
the Chinese remainder theorem, there exists an $a\in\mathcal{O}_{E}{}$ such
that $(a)=\mathfrak{a}{}\mathfrak{b}{}$ with $\mathcal{(\mathcal{O}}%
_{E}\mathcal{{}}\colon\mathfrak{a}{})$ and $(\mathcal{O}{}_{E}\colon
\mathfrak{b}{})$ relatively prime.\footnote{Take $a$ to be any element of
$\mathcal{O}{}_{E}$ satisfying an appropriate congruence condition for each
prime ideal $\mathfrak{p}{}$ of $\mathcal{O}{}_{E}$ such that $(\mathcal{O}%
{}_{E}\colon\mathfrak{p}{})$ is not prime to $(\mathcal{O}{}_{E}%
\colon\mathfrak{a}{})$.} Then
\[
\deg(\lambda^{\mathfrak{a}{}})\deg(\lambda^{\mathfrak{b}{}})=\deg
(\lambda^{(a)})=(\mathcal{O}_{E}\colon(a))=(\mathcal{O}_{E}\colon
\mathfrak{a)(\mathcal{O}}_{E}\mathfrak{\colon b).}%
\]
The only primes dividing $\deg(\lambda^{\mathfrak{a}{}})$ (resp. $\deg
(\lambda^{\mathfrak{b}{}})$) are those dividing $(\mathcal{O}_{E}%
\colon\mathfrak{a)}$ (resp. $(\mathcal{O}_{E}\colon\mathfrak{b)}$), and so we
must have $\deg(\lambda^{\mathfrak{a}{}})=(\mathcal{O}_{E}\colon\mathfrak{a)}$
and $\deg(\lambda^{\mathfrak{b}{}})=(\mathcal{O}\colon\mathfrak{b)}$.
\end{proof}

\begin{corollary}
\label{b36}Let $\mathfrak{a}{}$ be an invertible ideal in $R$. An $E$-isogeny
$\lambda\colon A\rightarrow B$ is an $\mathfrak{a}$-multiplication if and only
if $\deg(\lambda)=(R\colon\mathfrak{a)}$ and the maps $a\colon A\rightarrow A$
for $a\in\mathfrak{a}$ factor through $\lambda$.
\end{corollary}

\begin{proof}
We only have to prove the sufficiency of the conditions. According to the
definition (\ref{b30}), there exists an $E$-isogeny $\alpha\colon B\rightarrow
A^{\mathfrak{a}{}}$ such that $\alpha\circ\lambda=\lambda^{\mathfrak{a}{}}$.
Then $\deg(\alpha)\deg(\lambda)=\deg(\lambda^{\mathfrak{a}{}})$, and so
$\alpha$ is an isogeny of degree $1$, i.e., an isomorphism.
\end{proof}

\begin{proposition}
\label{b38}Let $E$ be a CM-algebra, and let $A$ and $B$ be abelian varieties
with complex multiplication by $E$ over $\mathbb{C}{}$. If $A$ and $B$ are
$E$-isogenous, then there exists a lattice ideal $\mathfrak{a}{}$ and an
$\mathfrak{a}{}$-multiplication $A\rightarrow B$.
\end{proposition}

\begin{proof}
Because $A$ and $B$ are $E$-isogenous, they have the same type $\Phi$. After
choosing $E$-basis elements for $H_{1}(A,\mathbb{Q}{})$ and $H_{1}%
(B,\mathbb{Q}{})$, we have isomorphisms%
\[
\mathbb{C}^{\Phi}/\Phi(\mathfrak{a}{})\rightarrow A(\mathbb{C}{}%
),\quad\mathbb{C}{}^{\Phi}/\Phi(\mathfrak{b}{})\rightarrow B(\mathbb{C}{}).
\]
Changing the choice of basis elements changes the ideals by principal ideals,
and so we may suppose that $\mathfrak{a}{}\subset\mathfrak{b}{}$. The quotient
map $\mathbb{C}{}^{\Phi}/\Phi(\mathfrak{a}{})\rightarrow\mathbb{C}{}^{\Phi
}/\Phi(\mathfrak{b}{})$ is an $\mathfrak{a}{}\mathfrak{b}{}^{-1}$-multiplication.
\end{proof}

\begin{proposition}
\label{b40}Let $A$ be an abelian variety with multiplication by $E$ over a
number field $k$, and assume that $A$ has good reduction at a prime
$\mathfrak{p}{}$ of $k$. The reduction to $k_{0}\overset{\text{{\tiny def}}%
}{=}\mathcal{O}{}_{k}/\mathfrak{p}$ of any $\mathfrak{a}{}$-multiplication
$\lambda\colon A\rightarrow B$ is again an $\mathfrak{a}{}$-multiplication.
\end{proposition}

\begin{proof}
Let $a_{1},\ldots,a_{n}$ be a basis for $\mathfrak{a}{}$, and let
$a_{i}=\alpha_{i}\circ\lambda$. In the diagram%
\[
\xymatrix{
A\ar[r]^\lambda\ar@/_1pc/[rr]_a&B\ar[r]^{\alpha}&A^n}\qquad\alpha={\left(
\begin{smallmatrix}
\alpha_{1}\\
\vdots\\
\alpha_{n}%
\end{smallmatrix}
\right)  }\quad a={\left(
\begin{smallmatrix}
a_{1}\\
\vdots\\
a_{n}%
\end{smallmatrix}
\right)  ,}%
\]
$\alpha$ maps $B$ isomorphically onto the image of $a$. Let $\mathcal{A}{}$
and $\mathcal{B}{}$ be abelian schemes over $\mathcal{O}{}_{\mathfrak{p}{}}$
with general fibre $A$ and $B$. Then the diagram extends uniquely to a diagram
over $\mathcal{O}{}_{\mathfrak{p}{}}$ (see (\ref{e3})), and reduces to a
similar diagram over $k_{0}$, which proves the proposition. (For an
alternative proof, see \ref{b41r}.)
\end{proof}

\subsection{$\mathfrak{a}{}$-multiplications: second approach}

In this subsection, $R$ is a commutative ring.

\begin{proposition}
\label{b41}Let $A$ be a commutative algebraic group $A$ over a field $k$ with
an action of $R$. For any finitely presented $R$-module $M$, the functor
\[
\underline{A}^{M}(T)=\Hom_{R}(M,A(T))\quad\quad\text{(}T\text{ a
}k\text{-scheme)}%
\]
is represented by a commutative algebraic group $A^{M}$ over $k$ with an
action of $R$. Moreover,%
\begin{equation}
A^{M\otimes_{R}N}\simeq(A^{M})^{N}. \label{e77}%
\end{equation}
If $M$ is projective and $A$ is an abelian variety, then $A^{M}$ is an abelian
variety (of dimension $r\dim A$ if $M$ is locally free of rank $r$).
\end{proposition}

\begin{proof}
If $M=R^{n}$, then $\underline{A}^{M}$ is represented by $A^{n}$. The functor
$M\mapsto\underline{A}^{M}$ transforms cokernels to kernels, and so a
presentation%
\[
R^{m}\rightarrow R^{n}\rightarrow M\rightarrow0,
\]
realizes $\underline{A}^{M}$ as a kernel%
\[
0\rightarrow\underline{A}^{M}\rightarrow A^{n}\rightarrow A^{m}.
\]
Define $A^{M}$ to be the kernel in the sense of algebraic groups.

For the second statement, use that there is an isomorphism of functors%
\[
\Hom_{R}(N,\Hom_{R}(M,A(T)))\simeq\Hom_{R}(M\otimes_{R}N,A(T)).
\]

For the final statement, if $M$ is projective, it is a direct summand of a
free $R$-module of finite rank. Thus $A^{M}$ is a direct factor of a product
of copies of $A$, and so is an abelian variety. Assume that $M$ is of constant
rank $r$. For an algebraic closure $\bar{k}$ of $k$ and a prime $\ell
\neq\mathrm{char}\,k,$%
\begin{align*}
A^{M}(\bar{k})_{\ell}  &  =\Hom_{R}(M,A(\bar{k})_{\ell})\\
&  \simeq\Hom_{R_{\ell}}(M_{\ell},A(\bar{k})_{\ell}),\quad R_{\ell}%
\overset{\text{{\tiny def}}}{=}\mathbb{Z}{}_{\ell}\otimes R,\quad M_{\ell
}\overset{\text{{\tiny def}}}{=}\mathbb{Z}{}_{\ell}\otimes_{\mathbb{Z}{}%
}M\text{.}%
\end{align*}
But $M_{\ell}$ is free of rank $r$ over $R_{\ell}$ (because $R$ is
semi-local), and so the order of $A^{M}(\bar{k})_{\ell}$ is $l^{2r\dim A}$.
Thus $A^{M}$ has dimension $r\dim A$.
\end{proof}

\begin{remark}
\label{b41r}The proposition (and its proof) applies over an arbitrary base
scheme $S$. Moreover, the functor $A\mapsto A^{M}$ commutes with base change
(because $A\mapsto\underline{A}^{M}$ obviously does). For example, if
$\mathcal{A}{}$ is an abelian scheme over the ring of integers $\mathcal{O}%
{}_{k}$ in a local field $k$ and $M$ is projective, then $\mathcal{A}{}^{M}$
is an abelian scheme over $\mathcal{O}{}_{k}$ with general fibre
$(\mathcal{A}{}_{k})^{M}$.
\end{remark}

\begin{proposition}
\label{b41a}Let $R$ act on an abelian variety $A$ over a field $k$. For any
finitely presented $R$-module $M$ and $\ell\neq\mathrm{char}\,k$,
\[
T_{\ell}(A^{M})\simeq\Hom_{R_{\ell}}(M_{\ell},T_{\ell}A),\quad R_{\ell
}\overset{\text{{\tiny def}}}{=}\mathbb{Z}{}_{\ell}\otimes R,\quad M_{\ell
}\overset{\text{{\tiny def}}}{=}\mathbb{Z}{}_{\ell}\otimes_{\mathbb{Z}{}%
}M\text{.}%
\]

\end{proposition}

\begin{proof}
As in the proof of (\ref{b41}),%
\[
A^{M}(\bar{k})_{\ell^{n}}\simeq\Hom_{R_{\ell}}(M_{\ell},A(\bar{k})_{\ell^{n}%
}).
\]
Now pass to the inverse limit over $n$.
\end{proof}

Let $R=\End_{R}(A)$. For any $R$-linear map $\alpha\colon M\rightarrow R$ and
$a\in A(T)$, we get an element%
\[
x\mapsto\alpha(x)\cdot a\colon M\rightarrow A(T)
\]
of \underline{$A$}$^{M}(T)$. In this way, we get a map $\Hom_{R}%
(M,R)\rightarrow\Hom_{R}(A,A^{M})$.

\begin{proposition}
\label{b41e}If $M$ is projective, then $\Hom_{R}(M,R)\simeq\Hom_{R}(A,A^{M}).$
\end{proposition}

\begin{proof}
When $M=R$, the map is simply $R\simeq\End_{R}(A)$. Similarly, when $M=R^{n}$,
the map is an isomorphism. In the general case, $M\oplus N\approx R^{n}$ for
some projective module $N$, and we have a commutative diagram%
\[
\begin{CD}
\Hom_{R}(M,R)\oplus\Hom_{R}(N,R) @>>> \Hom_{R}(A,A^{M}%
)\oplus\Hom_{R}(A,A^{N})\\
\wr@| \wr@|\\
\Hom_{R}(R^{n},R) @>\simeq>>\Hom_{R}(A,A^{n}).
\end{CD}
\]

\end{proof}

\begin{proposition}
\label{b41b}Let $A$ be an abelian variety over a field $k$, and let $R$ be a
commutative subring of $\End(A)$ such that $R\otimes_{\mathbb{Z}}\mathbb{Q}{}$
is a product of fields and $[R\colon\mathbb{Z}{}]=2\dim A$. For any invertible
ideal $\mathfrak{a}{}$ in $R$, the map $\lambda^{\mathfrak{a}{}}\colon
A\rightarrow A^{\mathfrak{a}{}}$ corresponding to the inclusion $\mathfrak{a}%
{}\hookrightarrow A$ is an isogeny with kernel $A_{\mathfrak{a}{}}%
\overset{\text{{\tiny def}}}{=}\bigcap\nolimits_{a\in\mathfrak{a}{}}\Ker(a{})$.
\end{proposition}

\begin{proof}
The functor $M\mapsto A^{M}$ sends cokernels to kernels, and so the exact
sequence%
\[
0\rightarrow\mathfrak{a}{}\rightarrow R\rightarrow R/\mathfrak{a}{}%
\rightarrow0
\]
gives rise to an exact sequence%
\[
0\rightarrow A^{R/\mathfrak{a}{}}\rightarrow A\overset{\lambda^{\mathfrak{a}%
{}}}{\longrightarrow}A^{\mathfrak{a}{}}.
\]
Clearly $A^{R/\mathfrak{a}{}}=A_{\mathfrak{a}{}}$, and so it remains to show
that $\lambda^{\mathfrak{a}{}}$ is surjective, but for a prime $\ell$ such
that $\mathfrak{a}{}_{\ell}=R_{\ell}$, $T_{\ell}(\lambda^{\mathfrak{a}{}})$ is
an isomorphism, from which this follows.
\end{proof}

\begin{corollary}
\label{b41c}Under the hypotheses of the proposition, the homomorphism
\[
\lambda^{\mathfrak{a}{}}\colon A\rightarrow A^{\mathfrak{a}{}}%
\]
corresponding to the inclusion $\mathfrak{a}{}\hookrightarrow R$ is an
$\mathfrak{a}{}$-multiplication.
\end{corollary}

\begin{proof}
A family of generators $(a_{i})_{1\leq i\leq n}$ for $\mathfrak{a}{}$ defines
an exact sequence%
\[
R^{m}\rightarrow R^{n}\rightarrow\mathfrak{a}{}\rightarrow0
\]
and hence an exact sequence%
\[
0\rightarrow A^{\mathfrak{a}{}}\rightarrow A^{n}\rightarrow A^{m}.
\]
The composite of
\[
R^{n}\rightarrow\mathfrak{a}{}\rightarrow R
\]
is $(r_{i})\mapsto\sum r_{i}a_{i}$, and so the composite of%
\[
A\overset{\lambda^{\mathfrak{a}{}}}{\longrightarrow}A^{\mathfrak{a}{}%
}\hookrightarrow A^{n}%
\]
is $x\mapsto(a_{i}x)_{1\leq i\leq n}$. As $\lambda^{\mathfrak{a}{}}$ is
surjective, it follows that $A^{\mathfrak{a}{}}$ maps onto the image of $A$ in
$A^{n}$, and so $\lambda^{\mathfrak{a}{}}$ is an $\mathfrak{a}{}%
$-multiplication (as shown in the proof of \ref{b32}).
\end{proof}

\begin{remark}
\label{b41d}Corollary \ref{b41c} fails if $\mathfrak{a}{}$ is not invertible.
Then $A^{\mathfrak{a}{}}$ need not be connected, $A\rightarrow(A^{\mathfrak{a}%
{}})^{\circ}$ is the $\mathfrak{a}{}$-multiplication, and $A^{\mathfrak{a}{}%
}/(A^{\mathfrak{a}{}})^{\circ}\simeq\Ext_{R}^{1}(R/\mathfrak{a}{},A)$ (see
\cite{waterhouse1969}, Appendix).
\end{remark}

\subsection{$\mathfrak{a}{}$-multiplications: complements}

Let $\lambda\colon A\rightarrow B$ be an $\mathfrak{a}{}$-multiplication, and
let $a\in\mathfrak{a}{}^{-1}\overset{\text{{\tiny def}}}{=}\{a\in E\mid
a\mathfrak{a}{}\in R\}$. Then $\lambda\circ a\in\Hom(A,B)$ (rather than
$\Hom^{0}(A,B)$). To see this, choose a basis for $a_{1},\ldots,a_{n}$ for
$\mathfrak{a}{}$, and note that the composite of the `homomorphisms'%
\[
A\overset{a}{\longrightarrow}A\xrightarrow{x\mapsto(\ldots,a_ix,\ldots)}A^{n}%
\]
is a homomorphism into $A^{\mathfrak{a}{}}\subset A^{n}$.

\begin{proposition}
\label{b46}Let $A$ have complex multiplication by $E$ over $k$.

(a) Let $\lambda\colon A\rightarrow B$ be an $\mathfrak{a}{}$-multiplication.
Then the map
\[
a\mapsto\lambda^{\mathfrak{a}{}}\circ a\colon\mathfrak{a}{}^{-1}%
\rightarrow\Hom_{R}(A,B)
\]
is an isomorphism. In particular, every $R$-isogeny $A\rightarrow B$ is a
$\mathfrak{b}{}$-multiplication for some ideal $\mathfrak{b}{}$.

(b) Assume $\mathcal{O}{}_{E}=\End(A)\cap E$. For any lattice ideals
$\mathfrak{a}{}\subset\mathfrak{b}{}$ in $\mathcal{O}{}_{E}$,
\[
\Hom_{\mathcal{O}{}_{E}}(A^{\mathfrak{a}{}},A^{\mathfrak{b}{}})\simeq
\mathfrak{a}{}^{-1}\mathfrak{b}{}\text{.}%
\]

\end{proposition}

\begin{proof}
(a) In view of (\ref{b41c}), the first statement is a special case of
(\ref{b41e}). For the second, recall (\ref{b31}) that $\lambda^{\mathfrak{a}%
{}}\circ a$ is an $\mathfrak{a}{}a$-multiplication.

(b) Recall that $A^{\mathfrak{b}{}}\simeq(A^{\mathfrak{a}{}})^{\mathfrak{a}%
{}^{-1}\mathfrak{b}{}}$ (see \ref{b34}), and so this follows from (a).
\end{proof}

In more down-to-earth terms, any two $E$-isogenies $A\rightarrow B$ differ by
an $E$-`isogeny'\ $A\rightarrow A$, which is an element of $E$. When $\lambda$
is an $\mathfrak{a}{}$-multiplication, the elements of $E$ such that
$\lambda\circ a$ is an isogeny (no quotes) are exactly those in $\mathfrak{a}%
{}^{-1}$.

\begin{proposition}
\label{b46a}Let $A$ have complex multiplication by $\mathcal{O}{}_{E}$ over an
algebraically closed field $k$ of characteristic zero. Then $\mathfrak{a}%
{}\mapsto A^{\mathfrak{a}{}}$ defines an isomorphism from the ideal class
group of $\mathcal{O}{}_{E}$ to the set of isogeny classes of abelian
varieties with complex multiplication by $\mathcal{O}{}_{E}$ over $k$ with the
same CM-type as $A$.
\end{proposition}

\begin{proof}
Proposition (\ref{b46}) shows that every abelian variety isogenous to $A$ is
an $\mathfrak{a}{}$-transform for some ideal $\mathfrak{a}{}$, and so the map
is surjective. As $a\colon A\rightarrow A$ is an $(a)$-multiplication,
principal ideals ideals map to $A$. Finally, if $A^{\mathfrak{a}{}}$ is
$\mathcal{O}{}_{E}$-isomorphic to $A$, then%
\[
\mathcal{O}{}_{E}\simeq\Hom_{\mathcal{O}{}_{E}}(A,A^{\mathfrak{a}{}}%
)\simeq\mathfrak{a}{}^{-1},
\]
and so $\mathfrak{a}{}$ is principal.
\end{proof}

\begin{proposition}
\label{b40c}Let $A$ and $B$ be abelian varieties with complex multiplication
by $\mathcal{O}{}_{E}$ over a number field $k$, and assume that they have good
reduction at a prime $\mathfrak{p}{}$ of $k$. If $A$ and $B$ are isogenous,
every $\mathcal{O}{}_{E}$-isogeny $\mu\colon A_{0}\rightarrow B_{0}$ lifts to
an $\mathfrak{a}{}$-multiplication $\lambda\colon A\rightarrow B$ for some
lattice ideal $\mathfrak{a}{}$, possibly after a finite extension of $k$. In
particular, $\mu$ becomes an $\mathfrak{a}{}$-multiplication over a finite
extension of $k$.
\end{proposition}

\begin{proof}
Since $A$ and $B$ are isogenous, there is an $\mathfrak{a}$-multiplication
$\lambda\colon A\rightarrow B$ for some lattice ideal $\mathfrak{a}{}$ by
(\ref{b38}) (after a finite extension of $k$). According to Proposition
\ref{b40}, $\lambda_{0}\colon A_{0}\rightarrow B_{0}$ is also an
$\mathfrak{a}{}$-multiplication. Hence the reduction map%
\[
\Hom_{\mathcal{O}{}_{E}}(A,B)\rightarrow\Hom_{\mathcal{O}{}_{E}}(A_{0},B_{0})
\]
is an isomorphism because both are isomorphic to $\mathfrak{a}{}^{-1}$, via
$\lambda$ and $\lambda_{0}$ respectively (\ref{b46}). Therefore, $\mu$ lifts
to an isogeny $\lambda^{\prime}\colon A\rightarrow B$, which is a
$\mathfrak{b}{}$-multiplication (see \ref{b46}).
\end{proof}

\clearpage

\section{The Shimura-Taniyama formula}

The \emph{numerical norm} of a nonzero integral ideal $\mathfrak{a}$ in a
number field $K$ is $\mathbb{N}{}\mathfrak{a}{}=(\mathcal{O}{}_{K}%
\colon\mathfrak{a}{})$. For a prime ideal $\mathfrak{p}{}$ lying over $p$,
$\mathbb{N}{}\mathfrak{p}{}=p^{f(\mathfrak{p}{}/p)}$. The map $\mathbb{N}{}$
is multiplicative: $\mathbb{N}{}\mathfrak{a}{}\cdot\mathbb{N}{}\mathfrak{b}%
{}=\mathbb{N}{}(\mathfrak{a}{}\mathfrak{b}{}).$

Let $k$ be a field of characteristic $p$, let $q$ be a power of $p$, and let
$\sigma$ be the homomorphism $a\mapsto a^{q}\colon k\rightarrow k$. For a
variety $V$ over $k$, we let $V^{(q)}=\sigma V$. For example, if $V$ is
defined by polynomials $\sum a_{i_{1}\cdots}X_{1}^{i_{1}}\cdots$, then
$V^{(q)}$ is defined by polynomials $\sum a_{i_{1}\cdots}^{q}X_{1}^{i_{1}%
}\cdots$. The $q$-\emph{power Frobenius map} is the regular map $V\rightarrow
\sigma V$ that acts by raising the coordinates of a $k^{\mathrm{al}}$-point of
$V$ to the $q$th power.

When $k=\mathbb{F}{}_{q}$, $V^{(q)}=V$ and the $q$-power Frobenius map is a
regular map $\pi\colon V\rightarrow V$. When $V$ is an abelian variety, the
Frobenius maps are homomorphisms.

\begin{theorem}
\label{b42}Let $A$ be an abelian variety with complex multiplication by a
CM-algebra $E$ over a number field $k$. Assume that $k$ contains all
conjugates of $E$ and let $\mathfrak{P}{}$ be prime ideal of $\mathcal{O}%
{}_{k}$ at which $A$ has good reduction. Assume (i) that $(p)\overset
{\text{{\tiny def}}}{=}\mathfrak{P}{}\cap\mathbb{Z}{}$ is unramified in $E$
and (ii) that $\End(A)\cap E=\mathcal{O}{}_{E}$.

\begin{enumerate}
\item There exists an element $\pi\in\mathcal{O}_{E}$ inducing the Frobenius
endomorphism on the reduction of $A$.

\item The ideal generated by $\pi$ factors as follows%
\begin{equation}
(\pi)=\prod\nolimits_{\varphi\in\Phi}\varphi^{-1}(\Nm_{k/\varphi
E}\mathfrak{P}{}) \label{e40}%
\end{equation}
where $\Phi\subset\Hom(E,k)$ is the CM-type of $A$.
\end{enumerate}
\end{theorem}

\begin{proof}
Let $A_{0}$ be the reduction of $A$ to $k_{0}\overset{\text{{\tiny def}}}%
{=}\mathcal{O}{}_{k}/\mathfrak{P}{}$, and let
\[
q=|k_{0}|=(\mathcal{O}{}_{k}\colon\mathfrak{P}{})=p^{f(\mathfrak{P}/p)}.
\]

(a) Recall that the reduction map $\End(A)\rightarrow\End(A_{0})$ is
injective. As $\End(A)\cap E$ is the maximal order $\mathcal{O}{}_{E}$ in $E$,
$\End(A_{0})\cap E$ must be also. The ($q$-power) Frobenius endomorphism $\pi$
of $A_{_{0}}$ commutes with all endomorphisms of $A_{0}$, and so it lies in
$\mathcal{O}{}_{E}$ by (\ref{b03b}).

(b) Let $\mathcal{A}{}$ be the abelian scheme over over $\mathcal{O}{}_{k}$
with fibres $A$ and $A_{0}$. Then $\mathcal{T}{}\overset{\text{{\tiny def}}%
}{=}\Tgt_{0}(\mathcal{A}{})$ is a free $\mathcal{O}{}_{k}$-module of rank
$\dim A$ such that $\mathcal{T}{}\otimes_{\mathcal{O}{}_{k}}k\simeq
\mathrm{Tgt}_{0}(A)$ and $\mathcal{T}/\mathfrak{P}{}\mathcal{T}{}%
\simeq\mathrm{Tgt}_{0}(A_{0})\overset{\text{{\tiny def}}}{=}T_{0}$.

Because $p$ is unramified in $E$, the isomorphism $E\otimes_{\mathbb{Q}{}%
}k\simeq\prod\nolimits_{\sigma\colon E\rightarrow k}k_{\sigma}$ induces an
isomorphism $\mathcal{O}{}_{E}\otimes_{\mathbb{Z}{}}\mathcal{O}{}_{k}%
\simeq\prod\nolimits_{\sigma\colon E\rightarrow k}\mathcal{O}{}_{\sigma}$
where $\mathcal{O}{}_{\sigma}$ denotes $\mathcal{O}{}_{k}$ with the
$\mathcal{O}{}_{E}$-algebra structure provided by $\sigma$. Similarly, the
isomorphism $T\simeq\bigoplus\nolimits_{\varphi\in\Phi}k_{\varphi}$ induces an
isomorphism $\mathcal{T}{}\simeq\bigoplus\nolimits_{\varphi\in\Phi}%
\mathcal{O}{}_{\varphi}$ where $\mathcal{O}{}_{\varphi}$ is the submodule of
$\mathcal{T}{}$ on which $\mathcal{O}{}_{k}$ acts through $\varphi$. In other
words, there exists an $\mathcal{O}{}_{k}$-basis $(e_{\varphi})_{\varphi
\in\Phi}$ for $\mathcal{T}{}$ such that $ae_{\varphi}=\varphi(a)e_{\varphi}$
for $a\in\mathcal{O}{}_{E}$.

Because $\pi\bar{\pi}=q$, the ideal $(\pi)$ is divisible only by prime ideals
dividing $p$, say,
\[
(\pi)=\prod\nolimits_{v|p}\mathfrak{p}{}_{v}^{m_{v}},\quad m_{v}\geq0.
\]
For $h$ the class number of $E$, let%
\begin{equation}
\mathfrak{p}_{v}^{m_{v}h}=(\gamma_{v}),\quad\gamma_{v}\in\mathcal{O}{}_{E},
\label{e44}%
\end{equation}
and let%
\begin{align*}
\Phi_{v}  &  =\{\varphi\in\Phi\mid\varphi^{-1}(\mathfrak{P})=\mathfrak{p}%
{}_{v}\},\\
d_{v}  &  =\left\vert \Phi_{v}\right\vert .
\end{align*}

The kernel of $T_{0}\overset{\gamma_{v}}{\longrightarrow}T_{0}$ is the span of
the $e_{\varphi}$ for which $\varphi(\gamma_{v})\in\mathfrak{P}{}$, but
$\varphi^{-1}(\mathfrak{P})$ is a prime ideal of $\mathcal{O}{}_{E}$ and
$\mathfrak{p}{}_{v}$ is the only prime ideal of $\mathcal{O}{}{}_{E}$
containing $\gamma_{v}$, and so $\varphi(\gamma_{v})\in\mathfrak{P}{}$ if and
only if $\varphi^{-1}(\mathfrak{P}{})=\mathfrak{p}{}_{v}$:%
\[
\Ker(T_{0}\overset{\gamma_{v}}{\longrightarrow}T_{0})=\langle e_{\varphi}%
\mid\varphi\in\Phi_{v}\rangle.
\]
Since $\pi^{h}\colon A_{0}\rightarrow A_{0}$ factors through $\gamma_{v}$, we
have that $\gamma_{v}^{\ast}k_{0}(A_{0})\supset(\pi^{h})^{\ast}k_{0}%
(A_{0})=k_{0}(A_{0})^{q^{h}}$, and so Proposition \ref{b18} shows that
\[
\deg(A_{0}\overset{\gamma_{v}}{\longrightarrow}A_{0})\leq q^{hd_{v}}.
\]
As%
\[
\deg(A_{0}\overset{\gamma_{v}}{\longrightarrow}A_{0})\overset{(\ref{b17})}%
{=}\mathbb{N(}{}\gamma_{v})\overset{(\ref{e44})}{=}\mathbb{N}{}(\mathfrak{p}%
{}_{v}^{hm_{v}})
\]
we deduce that%
\begin{equation}
\mathbb{N}{}(\mathfrak{p}_{v}^{m_{v}})\leq q^{d_{v}}. \label{e08}%
\end{equation}
On taking the product over $v$, we find that%
\[
\Nm_{E/\mathbb{Q}{}}(\pi)\leq q^{\sum\nolimits_{v|p}d_{v}}=q^{g}.
\]
But
\[
\Nm_{E/\mathbb{Q}{}}(\pi)\overset{(\ref{b17})}{=}\deg(A_{0}\overset{\pi
}{\longrightarrow}A_{0})=q^{g},
\]
and so the inequalities are all equalities.

Equality in (\ref{e08}) implies that
\[
\Nm_{E/\mathbb{Q}{}}(\mathfrak{p}{}_{v}^{m_{v}})=\left(  \Nm{}_{k/\mathbb{Q}%
{}}\mathfrak{P}{}\right)  ^{d_{v}}%
\]
which equals,%
\begin{align*}
\prod\nolimits_{\varphi\in\Phi_{v}}\Nm_{k/\mathbb{Q}{}}\mathfrak{P}  &
=\prod\nolimits_{\varphi\in\Phi_{v}}\left(  \Nm_{E/\mathbb{Q}{}}(\varphi
^{-1}(\Nm{}_{k/\varphi E}\mathfrak{P}{}))\right) \\
&  =\Nm{}_{E/\mathbb{Q}{}}\left(  \prod\nolimits_{\varphi\in\Phi_{v}}%
\varphi^{-1}(\Nm{}_{k/\varphi E}\mathfrak{P}{})\right)  .
\end{align*}
From the definition of $\Phi_{v}$, we see that $\prod\nolimits_{\varphi\in
\Phi_{v}}\varphi^{-1}(\Nm{}_{k/\varphi E}\mathfrak{P}{})$ is a power of
$\mathfrak{p}{}_{v}$, and so this shows that%
\begin{equation}
\mathfrak{p}{}_{v}^{m_{v}}=\prod\nolimits_{\varphi\in\Phi_{v}}\varphi
^{-1}(\Nm_{k/\varphi E}\mathfrak{P}{}). \label{e45}%
\end{equation}
On taking the product over $v$, we obtain the required formula.
\end{proof}

\begin{corollary}
\label{b42c}With the hypotheses of the theorem, for all primes $\mathfrak{p}%
{}$ of $E$ dividing $p$,%
\begin{equation}
\ord_{\mathfrak{p}{}}(\pi)=\sum\nolimits_{\varphi\in\Phi\text{, }\varphi
^{-1}(\mathfrak{P}{})=\mathfrak{p}{}}f(\mathfrak{P}{}/\varphi\mathfrak{p}{}).
\label{e46}%
\end{equation}
Here $\varphi\mathfrak{p}{}$ is the image of $\mathfrak{p}{}$ in
$\varphi\mathcal{O}{}_{E}\subset\varphi E\subset k$.
\end{corollary}

\begin{proof}
Let $\mathfrak{p}{}$ be a $p$-adic prime ideal of $\mathcal{O}{}_{E}$, and let
$\varphi$ be a homomorphism $E\rightarrow k$. If $\mathfrak{p}{}=\varphi
^{-1}(\mathfrak{P}{})$, then%
\[
\ord_{\mathfrak{p}{}}(\varphi^{-1}(\Nm_{k/\varphi E}\mathfrak{P}%
{}))=\ord_{\varphi\mathfrak{p}{}}\Nm_{k/\varphi E}\mathfrak{P}{}%
=f(\mathfrak{P}{}/\varphi\mathfrak{p}{}),
\]
and otherwise it is zero. Thus, (\ref{e46}) is nothing more than a restatement
of (\ref{e44}).
\end{proof}

\begin{corollary}
\label{b43}With the hypotheses of the theorem, for all primes $v$ of $E$
dividing $p$,%
\begin{equation}
\frac{\ord_{v}(\pi)}{\ord_{v}(q)}=\frac{|\Phi\cap H_{v}|}{|H_{v}|} \label{e48}%
\end{equation}
where $H_{v}=\{\rho\colon E\rightarrow k\mid\rho^{-1}(\mathfrak{\mathfrak{P}%
{}})=\mathfrak{p}_{v}\}$ and $q=(\mathcal{O}{}_{k}\colon\mathfrak{P}{})$.
\end{corollary}

\begin{proof}
We show that (\ref{e46}) implies (\ref{e48}) (and conversely) without assuming
$p$ to be unramified in $E$. Note that%
\[
\ord_{v}(q)=f(\mathfrak{P}{}/p)\cdot\ord_{v}(p)=f(\mathfrak{P}{}/p)\cdot
e(\mathfrak{p}_{v}/p),
\]
and that%
\[
|H_{v}|=e(\mathfrak{p}{}_{v}/p)\cdot f(\mathfrak{p}{}_{v}/p)\text{.}%
\]
Therefore, the equality
\[
\ord_{v}(\pi)=\sum\nolimits_{\varphi\in\Phi\cap H_{v}}f(\mathfrak{P}{}%
/\varphi\mathfrak{p}{}_{v}),
\]
implies that%
\[
\frac{\ord_{v}(\pi)}{\ord_{v}(q)}=\sum\nolimits_{\varphi\in\Phi\cap H_{v}%
}\frac{1}{e(\mathfrak{p}{}_{v}/p)\cdot f(\mathfrak{p}{}_{v}{}/p)}=|\Phi\cap
H_{v}|\cdot\frac{1}{|H_{v}|}%
\]
(and conversely).
\end{proof}

\begin{remark}
\label{b44r}(a) In the statement of Theorem \ref{b42}, $k$ can be replaced by
a finite extension of $\mathbb{Q}{}_{p}$.

(b) The conditions in the statement are unnecessarily strong. For example, the
formula holds without the assumption that $p$ be unramified in $E$. See
Theorem \ref{b57} below.

(c) When $E$ is a subfield of $k$, Theorem \ref{b42} can be stated in terms of
the reflex CM-type cf. \cite{shimuraT1961}, \S 13.
\end{remark}

\begin{application}
\label{b45}Let $A$ be an abelian variety with complex multiplication by a
CM-algebra $E$ over a number field $k{}$, and let $\Phi\subset\Hom(E,k)$ be
the type of $A$. Because $\Tgt_{0}(A)$ is an $E\otimes_{\mathbb{Q}{}}k$-module
satisfying (\ref{e01}), $k$ contains the reflex field $E^{\ast}$ of $(E,\Phi)$
and we assume $k$ is Galois over $E^{\ast}$. Let $\mathfrak{P}{}$ be a prime
ideal of $\mathcal{O}{}_{k}$ at which $A$ has good reduction, and let
$\mathfrak{\mathfrak{P}{}}\cap\mathcal{O}{}_{E^{\ast}}=\mathfrak{p}{}$ and
$\mathfrak{p}{}\cap\mathbb{Z}{}=(p)$. Assume

\begin{itemize}
\item that $p$ is unramified in $E$,

\item that $\mathfrak{p}{}$ is unramified in $k$, and

\item that $\End(A)\cap E=\mathcal{O}{}_{E}$.
\end{itemize}

\noindent Let $\sigma$ be the Frobenius element $(\mathfrak{P},k/E^{\ast})$ in
$\Gal(k/E^{\ast})$.\footnote{So $\sigma(\mathfrak{P}{})=\mathfrak{P}$ and
$\sigma a\equiv a^{p^{f(\mathfrak{P}{}/\mathfrak{p}{})}}\quad\mathrm{mod\,}%
\mathfrak{P}{}$ for all $a\in\mathcal{O}{}_{k}$.} As $\sigma$ fixes $E^{\ast}%
$, $A$ and $\sigma A$ have the same CM-type and so they become isogenous over
a finite extension of $k$. According to (\ref{b40c}), there exists an
$\mathfrak{a}{}$-multiplication $\lambda\colon A\rightarrow\sigma A$ over a
finite extension of $k$ whose reduction $\lambda_{0}\colon A_{0}\rightarrow
A_{0}^{(p^{f(\mathfrak{p}{}/p)})}$ is the $p^{f(\mathfrak{p}{}/p)}$-power
Frobenius map. Moreover,
\[
\sigma^{f(\mathfrak{P}{}/\mathfrak{p}{})-1}\lambda\circ\cdots\circ
\sigma\lambda\circ\lambda=\pi
\]
where $\pi$ is as in the statement of the theorem. Therefore, Theorem
\ref{b42} shows that%
\[
\mathfrak{a}{}^{f(\mathfrak{P}{}/\mathfrak{p}{})}=N_{\Phi}(\Nm_{k/E^{\ast}%
}\mathfrak{P}{})=N_{\Phi}(\mathfrak{p}{}^{f(\mathfrak{P}{}/\mathfrak{p}{}%
)})=N_{\Phi}(\mathfrak{p}{})^{f(\mathfrak{P}{}/\mathfrak{p}{})},
\]
and so%
\begin{equation}
\mathfrak{a}=N_{\Phi}(\mathfrak{\mathfrak{p}{}})\mathfrak{.} \label{e51}%
\end{equation}

Notice that, for any $m$ prime to $p$ and such that $A_{m}(k)=A_{m}%
(\mathbb{C)}{}$, the homomorphism $\lambda$ agrees with $\sigma$ on $A_{m}(k)$
(because it does on $A_{0,m}$).
\end{application}

\begin{nt}
The proof Theorem \ref{b42} in this section is essentially the original proof.
\end{nt}

\clearpage

\section{The fundamental theorem over the reflex field.}

\subsection{Preliminaries from algebraic number theory}

We review some class field theory (see, for example, \cite{milneCFT}, V). Let
$k$ be a number field. For a finite set $S$ of finite primes of $k$,
$I^{S}(k)$ denotes the group of fractional ideals of $k$ generated by the
prime ideals not in $S$. Assume $k$ is totally imaginary. Then a modulus for
$k$ is just an ideal in $\mathcal{O}{}_{k}$. For such a modulus $\mathfrak{m}%
{}$, $S(\mathfrak{m}{})$ denotes the set of finite primes $v$ dividing
$\mathfrak{m}{}$, and $k_{\mathfrak{m}{},1}$ denotes the group of $a\in
k^{\times}$ such that%
\[
\ord_{v}(a-1)\geq\ord_{v}(\mathfrak{m}{})
\]
for all finite primes $v$ dividing $\mathfrak{m}{}$. In other words, $a$ lies
in $k_{\mathfrak{m}{},1}$ if and only if multiplication by $a$ preserves
$\mathcal{O}{}_{v}\subset k_{v}$ for all $v$ dividing $\mathfrak{m}{}$ and
acts as $1$ on $\mathcal{O}{}_{v}/\mathfrak{p}{}_{v}^{\ord_{v}(\mathfrak{m}%
{})}=\mathcal{O}{}_{v}/\mathfrak{m}{}$. The ray class group modulo
$\mathfrak{m}{}$ is%
\[
C_{\mathfrak{m}{}}(k)=I^{S(\mathfrak{m}{})}/i(k_{\mathfrak{m}{},1})
\]
where $i$ is the map sending an element to its principal ideal. The
reciprocity map is an isomorphism%
\[
\mathfrak{a}{}\mapsto(\mathfrak{a}{},L_{\mathfrak{m}{}}/k)\colon
C_{\mathfrak{m}{}}(k)\rightarrow\Gal(L_{\mathfrak{m}{}}/k)
\]
where $L_{\mathfrak{m}{}}$ is the ray class field of $\mathfrak{m}{}$.

\begin{lemma}
\label{b60l}Let $\mathfrak{a}{}$ be a fractional ideal in $E$. For any integer
$m>0$, there exists an $a\in E^{\times}$ such that $a\mathfrak{a}{}%
\subset\mathcal{O}{}_{E}$ and $(\mathcal{O}_{E}{}\colon a\mathfrak{a}{})$ is
prime to $m$.
\end{lemma}

\begin{proof}
It suffices to find an $a\in E$ such that%
\begin{equation}
\ord_{v}(a)+\ord_{v}(\mathfrak{a}{})\geq0 \label{e61}%
\end{equation}
for all finite primes $v$, with equality holding if $v|m$.

Choose a $c\in\mathfrak{a}{}$. Then $\ord_{v}(c^{-1}\mathfrak{a}{})\leq0$ for
all finite $v$. For each $v$ such that $v|m$ or $\ord_{v}(\mathfrak{a}{})<0$,
choose an $a_{v}\in\mathcal{O}{}_{E}$ such that
\[
\ord_{v}(a_{v})+\ord_{v}(c^{-1}\mathfrak{a}{})=0
\]
(exists by the Chinese remainder theorem). For any $a\in\mathcal{O}{}_{E}$
sufficiently close to each $a_{v}$ (which exists by the Chinese remainder
theorem again), $ca$ satisfies the required condition.
\end{proof}

\subsection{The fundamental theorem in terms of ideals}

\begin{theorem}
\label{b57}Let $A$ be an abelian variety over $\mathbb{C}{}$ with complex
multiplication by a CM-algebra $E$, and let $\Phi\subset\Hom(E,\mathbb{C}{})$
be the type of $A$. Assume that $\End(A)\cap E=\mathcal{O}_{E}$. Fix an
integer $m>0$, and let $\sigma$ be an automorphism of $\mathbb{C}{}$ fixing
$E^{\ast}$.

\begin{enumerate}
\item There exists an ideal $\mathfrak{a}(\sigma){}$ in $E$ and an
$\mathfrak{a}{}(\sigma)$-multiplication $\lambda\colon A\rightarrow\sigma A$
such that $\lambda(x)=\sigma x$ for all $x\in A_{m}$; moreover, the class
$[\mathfrak{a}{}(\sigma)]$ of $\mathfrak{a}{}(\sigma)$ in $C_{m}(E)$ is
uniquely determined.

\item For any sufficiently divisible modulus $\mathfrak{m}{}$ for $E^{\ast}$,
the ideal class $[\mathfrak{a}{}(\sigma)]$ depends only on the restriction of
$\sigma$ to the ray class field $L_{\mathfrak{m}{}}$ of $\mathfrak{m}{}$, and%
\begin{equation}
\lbrack\mathfrak{a}(\sigma)]=[N_{\Phi}(\mathfrak{b})]{}\text{ if }%
\sigma|L_{\mathfrak{m}{}}=(\mathfrak{b}{},L_{\mathfrak{m}{}}/E^{\ast}).
\label{e02}%
\end{equation}

\end{enumerate}
\end{theorem}

\begin{proof}
Because $\sigma$ fixes $E^{\ast}$, the varieties $A$ and $\sigma A$ have the
same CM-types and so are $E$-isogenous. Therefore, there exists an
$\mathfrak{a}{}$-multiplication $\lambda\colon A\rightarrow\sigma A$ for some
ideal $\mathfrak{a}{}\subset\mathcal{O}{}_{E}$ (see \ref{b38}). Recall
(\ref{b35}) that $\lambda$ has degree $(\mathcal{O}{}_{E}\colon\mathfrak{a}%
{})$. After possibly replacing $\lambda$ with $\lambda\circ a$ for some
$a\in\mathfrak{a}{}^{-1}$, it will have degree prime to $m$ (apply
\ref{b60l}). Then $\lambda$ maps $A_{m}$\ isomorphically onto\ $\sigma A_{m}$.

Let $\mathbb{Z}{}_{m}=\prod_{\ell|m}\mathbb{Z}{}_{\ell}$ and $\mathcal{O}%
_{m}=\mathcal{O}{}_{E}\otimes\mathbb{Z}{}_{m}{}$. Then $T_{m}A\overset
{\text{{\tiny def}}}{=}\prod_{\ell|m}T_{\ell}A$ is a free $\mathcal{O}{}_{m}%
$-module of rank $1$ (see \ref{b03c}). The maps
\[%
\begin{array}
[c]{c}%
x\mapsto\sigma x\\
x\mapsto\lambda x
\end{array}
\colon T_{m}A\rightarrow T_{m}(\sigma A)
\]
are both $\mathcal{O}{}_{m}$-linear isomorphisms, and so they differ by a
homothety by an element $\alpha$ of $\mathcal{O}{}_{m}^{\times}$:%
\[
\lambda(\alpha x)=\sigma x,\quad\text{all }x\in T_{m}A.
\]
For any $a\in\mathcal{O}{}_{E}$ sufficiently close to $\alpha$, $\lambda\circ
a$ will agree with $\sigma$ on $A_{m}$. Thus, after replacing $\lambda$ with
$\lambda\circ a$, we will have%
\[
\lambda(x)\equiv\sigma x\mod
m,\quad\text{all }x\in T_{m}A.
\]
Now $\lambda$ is an $\mathfrak{a}{}$-multiplication for an ideal
$\mathfrak{a}=\mathfrak{a}{}(\sigma){}$ that is well-defined up to an element
of $i(E_{m,1})$.

Let $\sigma^{\prime}$ be a second element of $\Gal(\mathbb{C}/E^{\ast})$, and
let $\lambda^{\prime}\colon A\rightarrow\sigma^{\prime}A$ be an $\mathfrak{a}%
{}^{\prime}$-multiplication acting as $\sigma^{\prime}$ on $A_{m}$ (which
implies that it has degree prime to $m$). Then $\sigma\lambda^{\prime}$ is
again an $\mathfrak{a}{}^{\prime}$-multiplication (obvious from the definition
\ref{b30}), and so $\sigma\lambda^{\prime}\circ\lambda$ is an $\mathfrak{a}%
{}\mathfrak{a}{}^{\prime}$-multiplication $A\rightarrow\sigma^{\prime}\sigma
A$ (see \ref{b34}) acting as $\sigma^{\prime}\sigma$ on $A_{m}$. Therefore,
the map $\sigma\mapsto\mathfrak{a}{}(\sigma)\colon\Gal(\mathbb{C}/E^{\ast
})\rightarrow C_{m{}}(E)$ is a homomorphism, and so it factors through
$\Gal(k/E^{\ast}{})$ for some finite abelian extension $k$ of $E^{\ast}$,
which we may take to be the ray class field $L_{\mathfrak{m}{}}$. Thus, we
obtain a well-defined homomorphism%
\[
I^{S(\mathfrak{m}{})}(E^{\ast})\rightarrow C_{\mathfrak{m}{}}(E^{\ast
})\rightarrow C_{m}(E)
\]
sending an ideal $\mathfrak{a}{}^{\ast}$ in $I^{S(\mathfrak{m}{})}(E^{\ast})$
to $[\mathfrak{a}{}(\sigma)]$ where $\sigma=(\mathfrak{a}{}^{\ast
},L_{\mathfrak{m}{}}/E^{\ast})$. If $\mathfrak{m}{}$ is sufficiently
divisible, then $N_{\Phi}$ also defines a homomorphism $I^{S(\mathfrak{m}{}%
)}(E^{\ast})\rightarrow C_{\mathfrak{m}{}}(E^{\ast})\rightarrow C_{m}(E)$, and
it remains to show that the two homomorphisms coincide.

According to \S 1, there exists a field $k\subset\mathbb{C}{}$ containing the
ray class field $L_{\mathfrak{m}{}}$ and finite and Galois over $E^{\ast}$
such that $A$ has a model $A_{1}$ over $k$ with the following properties:

\begin{itemize}
\item $A_{1}$ has complex multiplication by $E$ over $k$ and $\mathcal{O}%
{}_{E}=\End(A_{1})\cap E$,

\item $A$ has good reduction at all the prime ideals of $\mathcal{O}{}_{k}$. and

\item $A_{m}(k)=A_{m}(\mathbb{C}{})$.
\end{itemize}

\noindent Now (\ref{b45}) shows that the two homomorphisms agree on the prime
ideals $\mathfrak{p}{}$ of $\mathcal{O}{}_{E^{\ast}}$ such $\mathfrak{p}{}$ is
unramified in $k$ and $\mathfrak{p}{}\cap\mathbb{Z}{}$ is unramified in $E$.
\noindent This excludes only finitely many prime ideals of $\mathcal{O}%
{}_{E^{\ast}}$, and according to Dirichlet's theorem on primes in arithmetic
progressions (e.g., \cite{milneCFT} V 2.5), the classes of these primes
exhaust $C_{\mathfrak{m}}$.
\end{proof}

\begin{aside}
\label{b56}Throughout this subsection and the next, $\mathbb{C}{}$ can be
replaced by an algebraic closure of $\mathbb{Q}{}$.
\end{aside}

\subsection{More preliminaries from algebraic number theory}

We let $\hat{\mathbb{Z}}=\plim{\mathbb{Z}}/m\mathbb{Z}$ and $\mathbb{A}%
_{f}=\hat{\mathbb{Z}}\otimes\mathbb{Q}$. For a number field $k$,
$\mathbb{A}_{f,k}=\mathbb{A}_{f}\otimes_{\mathbb{Q}{}}k$ is the ring of finite
ad\`{e}les and $\mathbb{A}_{k}=\mathbb{A}_{f,k}\times(k\otimes_{\mathbb{Q}{}%
}\mathbb{R})$ is the full ring of ad\`{e}les. For any ad\`{e}le $a$,
$a_{\infty}$ and $a_{f}$ denote its infinite and finite components. When $k$
is a subfield of $\mathbb{C}$, $k^{\mathrm{ab}}$ and $k^{\mathrm{al}}$ denote
respectively the largest abelian extension of $k$ in $\mathbb{C}$ and the
algebraic closure of $k$ in $\mathbb{C}$. As usual, complex conjugation is
denoted by $\iota$.

For a number field $k$, $\rec_{k}\colon\mathbb{A}_{k}^{\times}\rightarrow
\Gal(k^{\ab}/k)$ is the usual reciprocity law and $\mathrm{art}_{k}$ is its
reciprocal: a prime element corresponds to the inverse of the usual Frobenius
element. In more detail, if $a\in\mathbb{A}_{f,k}^{\times}$ has $v$-component
a prime element $a_{v}$ in $k_{v}$ and $w$-component $a_{w}=1$ for $w\neq v$,
then
\[
\mathrm{art}_{k}(a)(x)\equiv x^{1/\mathbb{N}(v)}\mod {\mathfrak{p}}_{v},\quad
x\in\mathcal{O}{}_{k}.
\]
${}$When $k$ is totally imaginary, $\mathrm{art}_{k}(a_{\infty}a_{f})$ depends
only on $a_{f}$, and we often regard $\mathrm{art}_{k}$ as a map
$\mathbb{A}_{f,k}^{\times}\rightarrow\Gal(k^{\mathrm{ab}}/k)$. Then%
\[
\mathrm{art}_{k}\colon\mathbb{A}_{f,k}^{\times}\rightarrow\Gal(k^{\mathrm{ab}%
}/k)
\]
is surjective with kernel the closure of $k^{\times}$ (embedded diagonally) in
$\mathbb{A}_{f,k}^{\times}$.

The cyclotomic character is the homomorphism $\chi_{\mathrm{cyc}}%
\colon\Aut(\mathbb{C}{})\rightarrow\mathbb{\hat{Z}}^{\times}{}{}%
\subset\mathbb{A}_{f}^{\times}$ such that $\sigma\zeta=\zeta^{\chi
_{\mathrm{cyc}}(\sigma)}$ for every root $\zeta$ of $1$ in $\mathbb{C}{}.$

\begin{lemma}
\label{b57a}For any $\sigma\in\Gal(\mathbb{Q}{}^{\mathrm{al}}/\mathbb{Q}{})$,%
\[
\mathrm{art}_{\mathbb{Q}{}}(\chi_{\mathrm{cyc}}(\sigma))=\sigma|\mathbb{Q}%
{}^{\mathrm{ab}}.
\]

\end{lemma}

\begin{proof}
Exercise (or see \cite{milneISV} \S 11 (50)).
\end{proof}

\begin{lemma}
\label{b57b}Let $E$ be a CM-field. For any $s\in\mathbb{A}_{f,E}^{\times}$ and
automorphism $\sigma$ of $\mathbb{C}$ such that $\mathrm{art}_{E}%
(s)=\sigma|E^{\mathrm{ab}}$,
\[
\Nm_{E/\mathbb{Q}{}}(s)\in\chi_{\mathrm{cyc}}(\mathrm{\sigma})\cdot
\mathbb{Q}{}_{>0}.
\]

\end{lemma}

\begin{proof}
By class field theory,%
\[
\mathrm{art}_{\mathbb{Q}{}}(\Nm_{E/\mathbb{Q}{}}(s))=\sigma|\mathbb{Q}%
{}^{\mathrm{ab}},
\]
which equals $\mathrm{art}_{\mathbb{Q}{}}(\chi_{\mathrm{cyc}}(\sigma))$.
Therefore $\Nm_{E/\mathbb{Q}{}}(s)$ and $\chi_{\mathrm{cyc}}(\mathrm{\sigma})$
differ by an element of the kernel of $\mathrm{art}_{\mathbb{Q}{}}%
\colon\mathbb{A}_{f}^{\times}\rightarrow\Gal(\mathbb{Q}{}^{\mathrm{ab}%
}/\mathbb{Q}{})$, which equals $\mathbb{A}_{f}^{\times}\cap\left(
\mathbb{Q}{}^{\times}\cdot\mathbb{R}{}_{>0}\right)  =$ $\mathbb{Q}{}_{>0}$
(embedded diagonally).
\end{proof}

\begin{lemma}
\label{b79}For any CM-field $E$, the kernel of $\mathrm{art}_{E}%
\colon\mathbb{A}_{f,E}^{\times}/E^{\times}\rightarrow\Gal(E^{\text{\textrm{ab}%
}}/E)$ is uniquely divisible by all integers, and its elements are fixed by
$\iota_{E}$.
\end{lemma}

\begin{proof}
The kernel of $\mathrm{art}_{E}$ is $\overline{E^{\times}}/E^{\times}$ where
$\overline{E^{\times}}$ is the closure of $E^{\times}$ in $\mathbb{A}%
_{f,E}^{\times}$. It is also equal to $\bar{U}/U$ for any subgroup $U$ of
$\mathcal{O}_{E}^{\times}$ of finite index. A theorem of Chevalley (see
\cite{serre1964s}, 3.5, or \cite{artinT1961}, Chap. 9, \S 1) shows that
$\mathbb{A}_{f,E}^{\times}$ induces the pro-finite topology on $U$. If we take
$U$ to be contained in the real subfield of $E$ and torsion-free, then it is
clear that $\bar{U}/U$ is fixed by $\iota_{E}$ and, being isomorphic to
$(\hat{\mathbb{Z}}/\mathbb{Z})^{[E\colon\mathbb{Q}{}]/2-1}$, it is uniquely divisible.
\end{proof}

\begin{lemma}
\label{b57e}Let $E$ be a CM-field and let $\Phi$ be a CM-type on $E$. For any
$s\in\mathbb{A}{}_{f,E}^{\times}$ and automorphism $\sigma$ of $\mathbb{C}$
such that $\mathrm{art}_{E}(s)=\sigma|E^{\mathrm{ab}}$,%
\[
N_{\Phi}(s)\cdot\iota_{E}N_{\Phi}(s)\in\chi_{\mathrm{cyc}}(\sigma
)\cdot\mathbb{Q}{}_{>0}.
\]

\end{lemma}

\begin{proof}
According to (\ref{e76}), p\pageref{e76},%
\[
N_{\Phi}(s)\cdot\iota_{E}N_{\Phi}(s)=\Nm_{E/\mathbb{Q}{}}(s),
\]
and so we can apply (\ref{b57b}).
\end{proof}

\begin{lemma}
\label{b57c}Let $E$ be a CM-field and $\Phi$ a CM-type on $E$. There exists a
unique homomorphism $\Gal(E^{\ast\mathrm{ab}}/E^{\ast})\rightarrow
\Gal(E^{\mathrm{ab}}/E)$ rendering%
\[
\begin{CD}
\mathbb{A}_{f,E^{\ast}}^{\times} @>{N_{\Phi}}>>\mathbb{A}_{f,E}^{\times}\\
@VV{\mathrm{art}_{E^{\ast}}}V@VV{\mathrm{art}_{E}}V\\
\Gal(E^{\ast\mathrm{ab}}/E^{\ast}) @>>> \Gal(E^{\mathrm{ab}}/E)
\end{CD}
\]
commutative.
\end{lemma}

\begin{proof}
As $\mathrm{art}_{E^{\ast}}\colon\mathbb{A}_{f,E^{\ast}}^{\times}%
\rightarrow\Gal(E^{\ast\mathrm{ab}}/E^{\ast})$ is surjective, the uniqueness
is obvious. On the other hand, $N_{\Phi}$ maps $E^{\ast\times}$ into
$E^{\times}$ and is continuous, and so it maps the closure of $E^{\ast\times}$
into the closure of $E^{\times}$.
\end{proof}

\begin{proposition}
\label{b57d}Let $s,s^{\prime}\in\mathbb{A}{}_{f,E^{\ast}}^{\times}$. If
$\mathrm{art}_{E^{\ast}}(s)=\mathrm{art}_{E^{\ast}}(s^{\prime})$, then
$N_{\Phi}(s^{\prime})\in N_{\Phi}(s)\cdot E^{\times}$.
\end{proposition}

\begin{proof}
Let $\sigma$ be an automorphism of $\mathbb{C}$ such that%
\[
\mathrm{art}_{E^{\ast}}(s)=\sigma|E^{\ast\mathrm{ab}}=\mathrm{art}_{E^{\ast}%
}(s^{\prime})\text{.}%
\]
Then (see \ref{b57e}),
\[
N_{\Phi}(s)\cdot\iota_{E}N_{\Phi}(s)\in\chi_{\text{cyc}}(\sigma)\cdot
\mathbb{Q}{}_{>0}\ni N_{\Phi}(s^{\prime})\cdot\iota_{E}N_{\Phi}(s^{\prime
})\text{.}%
\]
Let $t=N_{\Phi}(s/s^{\prime})\in\mathbb{A}_{f,E}^{\times}$. Then
$t\in\Ker(\mathrm{art}_{E})$ by (\ref{b57c}) and $t\cdot\iota_{E}%
t\in\mathbb{Q}{}_{>0}$. Lemma \ref{b79} implies that the map $x\mapsto
x\cdot\iota_{E}x$ is bijective on $\Ker(\mathrm{art}_{E})/E^{\times}$; as
$t\cdot\iota_{E}t\in E^{\times}$, so also does $t$.
\end{proof}

\subsection{The fundamental theorem in terms of id\`{e}les}

\begin{theorem}
\label{b58}Let $A$ be an abelian variety over $\mathbb{C}{}$ with complex
multiplication by a CM-algebra $E$, and let $\Phi\subset\Hom(E,\mathbb{C}{})$
be the type of $A$. Let $\sigma$ be an automorphism of $\mathbb{C}{}$ fixing
$E^{\ast}$. For any $s\in\mathbb{A}{}_{f,E^{\ast}}^{\times}$ with
$\mathrm{art}_{E^{\ast}}(s)=\sigma|E^{\ast\mathrm{ab}}$, there exists a unique
$E$-`isogeny'\ $\lambda\colon A\rightarrow\sigma A$ such that $\lambda
(N_{\Phi}(s)\cdot x)=\sigma x$ for all $x\in V_{f}A.$
\end{theorem}

\begin{remark}
\label{b59}(a) It is obvious that $\lambda$ is determined uniquely by the
choice of an $s$ such that $\rec(s)=\sigma|E^{\ast\mathrm{ab}}$. If $s$ is
replaced by $s^{\prime}$, then $N_{\Phi}(s^{\prime})=a\cdot N_{\Phi}(s)$ with
$a\in E^{\times}$ (see \ref{b57d}), and $\lambda$ must be replaced by
$\lambda\cdot a^{-1}$.

(b) The theorem is a statement about the $E$-`isogeny'\ class of $A$ --- if
$\beta\colon A\rightarrow B$ is an $E$-`isogeny', and $\lambda$ satisfies the
conditions of the theorem for $A$, then $\sigma\beta\circ\lambda\circ
\beta^{-1}$ satisfies the conditions for $B$. Therefore, in proving the
theorem we may assume that $\End(A)\cap E=\mathcal{O}{}_{E}$.

(c) Let $\lambda$ as in the theorem, let $\alpha$ be a polarization of $A$
whose Rosati involution induces $\iota_{E}$ on $E$, and let $\psi\colon
V_{f}A\times V_{f}A\rightarrow\mathbb{A}_{f}(1)$ be the Riemann form of
$\lambda$. The condition on the Rosati involution means that%
\begin{equation}
\psi(a\cdot x,y)=\psi(x,\iota_{E}a\cdot y),\quad x,y\in V_{f}A,\quad
a\in\mathbb{A}{}_{f,E}. \label{e1}%
\end{equation}
Then, for $x,y\in V_{f}A,$
\[
(\sigma\psi)(\sigma x,\sigma y)\overset{\text{{\tiny def}}}{=}\sigma
(\psi(x,y))=\chi_{\mathrm{cyc}}(\sigma)\cdot\psi(x,y)
\]
because $\psi(x,y)\in\mathbb{A}_{f}(1)$. Thus if $\lambda$ is as in the
theorem, then
\begin{equation}
\chi_{\mathrm{cyc}}(\sigma)\cdot\psi(x,y)=(\sigma\psi)(N_{\Phi}(s)\lambda
(x),N_{\Phi}(s)\lambda(y)). \label{e2}%
\end{equation}
According to (\ref{e76}), p\ref{e76}, $N_{\Phi}(s)\cdot\iota_{E}N_{\Phi
}(s)=N_{E^{\ast}/\mathbb{Q}{}}(s)$, and so, on combining (\ref{e1}) and
(\ref{e2}), we find that
\[
(c\psi)(x,y)=(\sigma\psi)(\lambda x,\lambda y),
\]
with $c=\chi_{\mathrm{cyc}}(\sigma)/N_{E^{\ast}/\mathbb{Q}}(s)\in\mathbb{Q}%
{}_{>0}$ (see \ref{b57b}).
\end{remark}

Let $\sigma{}$ be an automorphism of $\mathbb{C}{}$ fixing $E^{\ast}$. Because
$\sigma$ fixes $E^{\ast}$, there exists an $E$-isogeny $\lambda\colon
A\rightarrow\sigma A$. The maps
\[%
\begin{array}
[c]{c}%
x\mapsto\sigma x\\
x\mapsto\lambda x
\end{array}
\text{ }\colon V_{f}(A)\rightarrow V_{f}(\sigma A)
\]
are both $\mathbb{A}{}_{f,E}\overset{\text{{\tiny def}}}{=}E\otimes
\mathbb{A}_{f}$-linear isomorphisms. As $V_{f}(A)$ is free of rank one over
$\mathbb{A}{}_{f,E}$,\footnote{This is even true when $R\overset
{\text{{\tiny def}}}{=}\End(A)\cap E$ is not the whole of $\mathcal{O}{}_{E}$
because, for all $\ell$, $V_{\ell}A$ is free of rank one over $E_{\ell
}\overset{\text{{\tiny def}}}{=}E\otimes_{\mathbb{Q}{}}\mathbb{Q}{}_{\ell}$,
and for all $\ell$ not dividing $(\mathcal{O}{}_{E}\colon R)$, $T_{\ell}A$ is
free of rank one over $R_{\ell}\overset{\text{{\tiny def}}}{=}R\otimes
_{\mathbb{Z}{}}\mathbb{Z}{}_{\ell}$ (see \ref{b03c}).} they differ by a
homothety by an element $\eta(\sigma)$ of $\mathbb{A}{}_{f,E}^{\times}$:%
\begin{equation}
\lambda(\eta(\sigma)x)=\sigma x,\quad\text{all }x\in V_{f}(A).\label{e70}%
\end{equation}
When the choice of $\lambda$ is changed, $\eta(\sigma)$ is changed only by an
element of $E^{\times}$, and so we have a well-defined map
\begin{equation}
\Aut(\mathbb{C}{}/E^{\ast})\rightarrow\mathbb{A}{}_{f,E}^{\times}/E^{\times
}.\label{e71}%
\end{equation}

\begin{lemma}
\label{b58c}For a suitable choice of $\lambda$, the quotient $t=\eta
(\sigma)/N_{\Phi}(s)$ satisfies the equation $t\cdot\iota_{E}t=1$ in
$\mathbb{A}{}_{f,E}^{\times}$.
\end{lemma}

\begin{proof}
We know that
\[
N_{\Phi}(s)\cdot\iota_{E}N_{\Phi}(s)\overset{(\ref{e76})}{=}\Nm_{\mathbb{A}%
_{f,E^{\ast}}/\mathbb{A}{}_{f}}(s)\overset{(\ref{b57b})}{=}\chi_{\mathrm{cyc}%
}(\sigma)\cdot a
\]
for some $a\in\mathbb{Q}{}_{>0}$.

A calculation as in (\ref{b59}c) shows that,%
\begin{equation}
(c\psi)(x,y)=(\sigma\psi)(\lambda x,\lambda y),\text{ all }x,y\in\mathbb{A}%
{}_{f,E}^{\times}, \label{e67}%
\end{equation}
with $c=\chi_{\mathrm{cyc}}(\sigma)/\left(  \eta(\sigma)\cdot\iota_{E}%
\eta(\sigma)\right)  $. Now the discussion on p\pageref{riemann} shows that
$c$ is a totally positive element of $F$. Thus%
\begin{equation}
\eta(\sigma)\cdot\iota_{E}\eta(\sigma)=\chi_{\mathrm{cyc}}(\sigma)/c,\quad
c\in F_{\gg0}. \label{e68}%
\end{equation}

Let $t=\eta(\sigma)/N_{\Phi}(s)$. Then%
\begin{equation}
t\cdot\iota_{E}t=1/ac\in F_{\gg0}. \label{e69}%
\end{equation}
Being a totally positive element of $F$, $ac$ is a local norm from $E$ at the
infinite primes, and (\ref{e69}) shows that it is also a local norm at the
finite primes. Therefore we can write $ac=e\cdot\iota_{E}e$ for some $e\in
E^{\times}$. Then%
\[
te\cdot\iota_{E}(te)=1.
\]

\end{proof}

The map $\eta\colon\Gal(\mathbb{Q}^{\mathrm{al}}/E^{\ast})\rightarrow
\mathbb{A}{}_{f,E}^{\times}/E^{\times}$ is a homomorphism\footnote{Choose
$E$-isogenies $\alpha\colon A\rightarrow\sigma A$ and $\alpha^{\prime}\colon
A\rightarrow\sigma^{\prime}A^{\prime}$, and let%
\begin{align*}
\alpha(sx)  &  =\sigma x\\
\alpha^{\prime}(s^{\prime}x)  &  =\sigma^{\prime}x.
\end{align*}
Then $\sigma\alpha^{\prime}\circ\alpha$ is an isogeny $A\rightarrow
\sigma\sigma^{\prime}A$, and%
\[
(\sigma\alpha^{\prime}\circ\alpha)(ss^{\prime}x)=(\sigma\alpha^{\prime
})(\alpha(ss^{\prime}x))=(\sigma\alpha^{\prime})(\sigma(s^{\prime}%
x))=\sigma(\alpha^{\prime}(s^{\prime}x))=\sigma\sigma^{\prime}x.
\]
}, and so it factors through $\Gal(\mathbb{Q}{}^{\text{al}}/E^{\ast
})^{\text{ab}}$. When combined with the Artin map, it gives a homomorphism
$\eta^{\prime}\colon\mathbb{A}{}_{f,E^{\ast}}^{\times}/E^{\ast\times
}\rightarrow\mathbb{A}{}_{f,E}^{\times}/E^{\times}$.\noindent

Choose an integer $m>0$. For some modulus $\mathfrak{m}{}$, there exist a
commutative diagrams%
\[
\begin{CD}
\mathbb{A}_{f,E^{\ast}}^{\times}/E^{\ast\times} @>>> \mathbb{A}_{f,E}^{\times}/E^{\times}\\
@VV{\mathrm{onto}}V@VV{\mathrm{onto}}V\\
C_{\mathfrak{m}}(E^{\ast})@>>>C_{m}(E)
\end{CD}
\]
with the top map either $N_{\Phi}{}$ or by $\eta^{\prime}$ and the vertical
maps the obvious maps (\cite{milneCFT}, V 4.6). The bottom maps in the two
diagrams agree by Theorem \ref{b57}, which implies that $t\overset
{\text{{\tiny def}}}{=}\eta(\sigma)/N_{\Phi}(s)$ lies in the common kernel of
the maps $\mathbb{A}{}_{f,E}^{\times}/E^{\times}\rightarrow C_{m}(E)$ for
$m>0$. But this common kernel is equal to the kernel of the Artin map
$\mathbb{A}{}_{f,E}^{\times}/E^{\times}\rightarrow{}\Gal(E^{\mathrm{ab}}/E)$.
As $t\cdot\iota_{E}t=1$ (see \ref{b58c}), $t=1$ (see \ref{b79}).

\begin{nt}
The proof in this subsection is that sketched in \cite{milneISV}. Cf.
\cite{shimura1970}, 4.3, and \cite{shimura1971}, pp117--121, p129.
\end{nt}

\subsection{The fundamental theorem in terms of uniformizations}

Let $(A,i\colon E\hookrightarrow\End^{0}(A))$ be an abelian variety with
complex multiplication over $\mathbb{C}{}$, and let $\alpha$ be a polarization
of $(A,i)$. Recall (\ref{a30p}, \ref{a34}) that the choice of a basis element
$e_{0}$ for $H_{0}(A,\mathbb{Q}{})$ determines a uniformization $\theta
\colon\mathbb{C}{}^{\Phi}\rightarrow A(\mathbb{C}{})$, and hence a quadruple
$(E,\Phi;\mathfrak{a}{},t)$, called the type of $(A,i,\lambda)$ relative to
$\theta$.

\begin{theorem}
\label{b61}Let $(A,i,\lambda)$ be of type $(E,\Phi;\mathfrak{a,}t\mathfrak{)}$
relative to a uniformization $\theta\colon\mathbb{C}^{\Phi}\rightarrow
A(\mathbb{C})$, and let $\sigma$ be an automorphism of $\mathbb{C}{}$ fixing
$E^{\ast}$. For any $s\in\mathbb{A}_{f,E^{\ast}}^{\times}$ such that
$\mathrm{art}_{E^{\ast}}(s)=\sigma|E^{\ast\text{ab}}$, there is a unique
uniformization $\theta^{\prime}\colon\mathbb{C}^{\Phi^{\prime}}\rightarrow
(\sigma A)(\mathbb{C})$ of $\sigma A$ such that

\begin{enumerate}
\item $\sigma(A,i,\psi)$ has type $(E,\Phi;f\mathfrak{a},t\cdot\chi
_{\mathrm{cyc}}(\sigma)/f\bar{f})$ where $f=N_{\Phi}(s)\in\mathbb{A}%
_{f,E}^{\times};$

\item the diagram%
\[
\begin{CD}
E/\mathfrak{a}@>{\theta_{0}}>> A(\mathbb{C})\\
@VVfV@VV{\sigma}V\\
E/f\mathfrak{a}@>{\theta_{0}^{\prime}}>>\sigma A(\mathbb{C})
\end{CD}
\]
commutes, where $\theta_{0}(x)=\theta((\varphi x)_{\varphi\in\Phi})$ and
$\theta_{0}^{\prime}(x)=\theta^{\prime}((\varphi x)_{\varphi\in\Phi^{\prime}%
}).$
\end{enumerate}
\end{theorem}

\begin{proof}
According to Theorem \ref{b58}, there exists an isogeny $\lambda\colon
A\rightarrow\sigma A$ such that $\lambda(N_{\Phi}(s)\cdot x)=\sigma x$ for all
$x\in V_{f}A$. Then $H_{1}(\lambda)$ is an $E$-linear isomorphism
$H_{1}(A,\mathbb{Q}{})\rightarrow H_{1}(\sigma A,\mathbb{Q}{})$, and we let
$\theta^{\prime}$ be the uniformization defined by the basis element
$H_{1}(\lambda)(e_{0})$ for $H_{1}(\sigma A,\mathbb{Q}{})$. The statement now
follows immediately from Theorem \ref{b58} and (\ref{b59}c).
\end{proof}

\clearpage

\section{The fundamental theorem over $\mathbb{Q}{}$}

The first three subsections follow \cite{tate1981} and the last subsection
follows \cite{deligne1981}.

We begin by reviewing some notations. We let $\hat{\mathbb{Z}}%
=\plim{\mathbb{Z}}/m\mathbb{Z}$ and $\mathbb{A}_{f}=\hat{\mathbb{Z}}%
\otimes\mathbb{Q}$. For a number field $k$, $\mathbb{A}_{f,k}=\mathbb{A}%
_{f}\otimes_{\mathbb{Q}{}}k$ is the ring of finite ad\`{e}les and
$\mathbb{A}_{k}=\mathbb{A}_{f,k}\times(k\otimes_{\mathbb{Q}{}}\mathbb{R})$ is
the full ring of ad\`{e}les. When $k$ is a subfield of $\mathbb{C}$,
$k^{\mathrm{ab}}$ and $k^{\mathrm{al}}$ denote respectively the largest
abelian extension of $k$ in $\mathbb{C}$ and the algebraic closure of $k$ in
$\mathbb{C}$. For a number field $k$, $\rec_{k}\colon\mathbb{A}_{k}^{\times
}\rightarrow\Gal(k^{\ab}/k)$ is the usual reciprocity law and $\mathrm{art}%
_{k}$ is its reciprocal. When $k$ is totally imaginary, we also write
$\mathrm{art}_{k}$ for the map $\mathbb{A}_{f,k}^{\times}\rightarrow
\Gal(k^{\mathrm{ab}}/k)$ that it defines. The cyclotomic character $\chi
=\chi_{\text{\textrm{cyc}}}\colon\Aut(\mathbb{C})\rightarrow\hat{\mathbb{Z}%
}^{\times}\subset\mathbb{A}_{f}^{\times}$ is the homomorphism such that
$\sigma\zeta=\zeta^{\chi(\sigma)}$ for every root of $1$ in $\mathbb{C}$. The
composite
\begin{equation}
\mathrm{art}_{k}\circ\chi_{\mathrm{cyc}}=\mathrm{Ver}_{k/\mathbb{Q}%
},\label{e80}%
\end{equation}
the Verlagerung map $\Gal(\mathbb{Q}^{\mathrm{al}}/\mathbb{Q})^{\mathrm{ab}%
}\rightarrow\Gal(\mathbb{Q}^{\mathrm{al}}/k)^{\mathrm{ab}}$.

\subsection{Statement of the Theorem}

Let $A$ be an abelian variety over $\mathbb{C}$, and let $E$ be a subfield of
$\End(A)\otimes\mathbb{Q}$ of degree $2\dim A$ over $\mathbb{Q}$. The
representation of $E$ on the tangent space to $A$ at zero is of the form
$\bigoplus_{\varphi\in\Phi}\varphi$ with $\Phi$ a subset of $\Hom(E,\mathbb{C}%
)$. A \emph{Riemann form} for $A$ is a $\mathbb{Q}$-bilinear skew-symmetric
form $\psi$ on $H_{1}(A,\mathbb{Q})$ such that
\[
(x,y)\mapsto\psi(x,iy)\colon H_{1}(A,\mathbb{R})\times H_{1}(A,\mathbb{R}%
)\rightarrow\mathbb{R}%
\]
is symmetric and positive definite. We assume that there exists a Riemann form
$\psi$ compatible with the action of $E$ in the sense that, for some
involution $\iota_{E}$ of $E$,
\[
\psi(ax,y)=\psi(x,(\iota_{E}a)y),\quad a\in E,\quad x,y\in H_{1}%
(A,\mathbb{Q}).
\]
Then $E$ is a CM-field, and $\Phi$ is a CM-type on $E$, i.e.,
$\Hom(E,\mathbb{C})=\Phi\cup\iota\Phi$ (disjoint union). The pair
$(A,E\hookrightarrow\End(A)\otimes\mathbb{Q})$ is said to be of \emph{CM-type}
$(E,\Phi)$. For simplicity, we assume that $E\cap\End(A)=\mathcal{O}_{E}$, the
full ring of integers in $E$.

Let $\mathbb{C}^{\Phi}$ be the set of complex-valued functions on $\Phi$,\/
and embed $E$ into $\mathbb{C}^{\Phi}$ through the natural map $a\mapsto
(\varphi(a))_{\varphi\in\Phi}$. There then exist a $\mathbb{Z}$-lattice
$\mathfrak{a}$ in $E$ stable under $\mathcal{O}_{E}$, an element $t\in
E^{\times}$, and an $\mathcal{O}_{E}$-linear analytic isomorphism
$\theta\colon\mathbb{C}^{\Phi}/\Phi(\mathfrak{a)}\rightarrow A$ such that
$\psi(x,y)=\Tr_{E/\mathbb{Q}}(tx\cdot\iota_{E}y)$ where, in the last equation,
we have used $\theta$ to identify $H_{1}(A,\mathbb{Q})$ with $\mathfrak{a}%
\otimes\mathbb{Q}=E$. The variety is said to be of \emph{type} $(E,\Phi
;\mathfrak{a},t)$ relative to $\theta$. The type determines the triple
$(A,E\hookrightarrow\End(A)\otimes\mathbb{Q},\psi)$ up to isomorphism.
Conversely, the triple determines the type up to a change of the following
form: if $\theta$ is replaced by $\theta\circ a^{-1}$, $a\in E^{\times}$, then
the type becomes $(E,\Phi;a\mathfrak{a},\frac{t}{a\cdot\iota a})$ (see
\ref{a34}).

Let $\sigma\in\Aut(\mathbb{C})$. Then $E\hookrightarrow\End^{0}(A)$ induces a
map $E\hookrightarrow\End^{0}(\sigma A)$, so that $\sigma A$ also has complex
multiplication by $E$. The form $\psi$ is associated with a divisor $D$ on
$A$, and we let $\sigma\psi$ be the Riemann form for $\sigma A$ associated
with $\sigma D$. It has the following characterization: after multiplying
$\psi$ with a nonzero rational number, we can assume that it takes integral
values on $H_{1}(A,\mathbb{Z})$; define $\psi_{m}$ to be the pairing
$A_{m}\times A_{m}\rightarrow\mu_{m}$, $(x,y)\mapsto\exp(\frac{2\pi i\cdot
\psi(x,y)}{m})$; then $(\sigma\psi)_{m}(\sigma x,\sigma y)=\sigma(\psi
_{m}(x,y))$ for all $m$.

In the next section we shall define for each CM-type $(E,\Phi)$ a map
$f_{\Phi}\colon\Aut(\mathbb{C})\rightarrow\mathbb{A}_{f,E}^{\times}/E^{\times
}$ such that
\[
f_{\Phi}(\sigma)\cdot\iota f_{\Phi}(\sigma)=\chi_{\text{cyc}}(\sigma
)E^{\times},\quad\text{\textrm{all }}\sigma\in\Aut(\mathbb{C}).
\]
We can now state the fundamental theorem of complex multiplication.

\begin{theorem}
\label{b72}Suppose $A$ has type $(E,\Phi;{\mathfrak{a}},t)$ relative to the
uniformization $\theta\colon\mathbb{C}^{\Phi}/\mathfrak{a}{}\rightarrow A$.
Let $\sigma\in\Aut(\mathbb{C})$, and let $f\in\mathbb{A}_{f,E}^{\times}$ lie
in $f_{\Phi}(\sigma)$.

\begin{enumerate}
\item The variety $\sigma A$ has type
\[
(E,\sigma\Phi;f{\mathfrak{a}},\frac{t\chi_{\text{cyc}}(\sigma)}{f\cdot\iota
f})
\]
relative some uniformization $\theta^{\prime}$.

\item It is possible to choose $\theta^{\prime}$ so that%
\[
\begin{CD}
\mathbb{A}_{f,E} @>>>\mathbb{A}_{f,E}/{\mathfrak{a}}\otimes\hat{{\mathbb{Z}}}\simeq{E}/{\mathfrak{a}}
@>{\theta}>>A_{\text{tors}}\\
@VV{f}V@.@VV{\sigma}V\\
{\mathbb{A}}_{f,E} @>>> \mathbb{A}_{f,E}/(f{\mathfrak{a}}\otimes\hat{{\mathbb{Z}}})
\simeq{E}/f{\mathfrak{a}}@>{\theta'}>>\sigma A_{\text{tors}}
\end{CD}
\]
commutes, where $A_{\text{\textrm{tors}}}$ denotes the torsion subgroup of $A$
(and then $\theta^{\prime}$ is uniquely determined),
\end{enumerate}
\end{theorem}

We now restate the theorem in a more canonical form. Let
\[
TA\overset{\text{{\tiny def}}}{=}\plim A_{m}(\mathbb{C})\simeq\plim(\tfrac
{1}{m}H_{1}(A,\mathbb{Z})/H_{1}(A,\mathbb{Z}))\simeq H_{1}(A,\hat{\mathbb{Z}%
})
\]
(limit over all positive integers $m$), and let
\[
V_{f}A\overset{\text{{\tiny def}}}{=}TA\otimes_{\mathbb{Z}}\mathbb{Q}\simeq
H_{1}(A,\mathbb{Q})\otimes_{\mathbb{Q}}\mathbb{A}_{f}.
\]
Then $\psi$ gives rise to a pairing
\[
\psi_{f}=\plim\psi_{m}\colon V_{f}A\times V_{f}A\rightarrow\mathbb{A}_{f}(1)
\]
where $\mathbb{A}_{f}(1)=(\plim\mu_{m}(\mathbb{C}))\otimes\mathbb{Q}$.

\begin{theorem}
\label{b74}Let $A$ have type $(E,\Phi)$; let $\sigma\in\Aut(\mathbb{C})$, and
let $f\in f_{\Phi}(\sigma)$.

\begin{enumerate}
\item $\sigma A$ is of type $(E,\sigma\Phi)$;

\item there is an $E$-linear isomorphism $\alpha\colon H_{1}(A,\mathbb{Q}%
)\rightarrow H_{1}(\sigma A,\mathbb{Q})$ such that

\begin{enumerate}
\item $\psi(\frac{\chi_{\mathrm{cyc}}(\sigma)}{f\cdot\iota f}x,y)=(\sigma
\psi)(\alpha x,\alpha y),\quad x,y\in H_{1}(A,\mathbb{Q})$;

\item \ the\footnote{Note that both $f\in\mathbb{A}_{f,E}^{\times}$ and the
$E$-linear isomorphism $\alpha$ are uniquely determined up to multiplication
by an element of $E^{\times}$. Changing the choice of one changes that of the
other by the same factor.} diagram
\[
\xymatrix{
V_f(A)\ar[r]^{f}\ar[rd]_{\sigma}&V_f(A)\ar[d]^{\alpha\otimes 1}\\
&V_f(\sigma{A})
}
\]
commutes.
\end{enumerate}
\end{enumerate}
\end{theorem}

\begin{lemma}
\label{b75}The statements (\ref{b72}) and (\ref{b74}) are equivalent.
\end{lemma}

\begin{proof}
Let $\theta$ and $\theta^{\prime}$ be as in (\ref{b72}), and let $\theta
_{1}\colon E\xrightarrow {\approx}H_{1}(A,\mathbb{Q})$ and $\theta_{1}%
^{\prime}\colon E\xrightarrow {\approx}H_{1}(\sigma A,\mathbb{Q})$ be the
$E$-linear isomorphisms induced by $\theta$ and $\theta^{\prime}$. Let
$\chi=\chi_{\text{cyc}}(\sigma)/f\cdot\iota f$ --- it is an element of
$E^{\times}$. Then
\begin{align*}
\psi(\theta_{1}(x),\theta_{1}(y))  &  =\Tr_{E/\mathbb{Q}{}}(tx\cdot\iota y)\\
(\sigma\psi)(\theta_{1}^{\prime}(x),\theta_{1}^{\prime}(y))  &
=\Tr_{E/\mathbb{Q}{}}(t\chi x\cdot\iota y)
\end{align*}
and
\[
\begin{CD}
{\mathbb{A}}_{f,E}@>{\theta_1}>>V_f(A)\\
@VV{f}V@VV{\sigma}V\\
{\mathbb{A}}_{f,E}@>\theta_1'>>V_f(\sigma{A})
\end{CD}
\]
commutes. Let $\alpha=\theta_{1}^{\prime}\circ\theta_{1}^{-1}$; then
\[
(\sigma\psi)(\alpha x,\alpha y)=\Tr_{E/\mathbb{Q}{}}(t\chi\theta_{1}%
^{-1}(x)\cdot\iota\theta_{1}^{-1}(y))=\psi(\chi x,y)
\]
and (on $V_{f}(A)$),
\[
\sigma=\theta_{1}^{\prime}\circ f\circ\theta_{1}^{-1}=\theta_{1}^{\prime}%
\circ\theta_{1}^{-1}\circ f=\alpha\circ f.
\]
Conversely, let $\alpha$ be as in (\ref{b74}) and choose $\theta_{1}^{\prime}$
so that $\alpha=\theta_{1}^{\prime}\circ\theta_{1}^{-1}$. The argument can be
reversed to deduce (\ref{b72}).
\end{proof}

\subsection{Definition of $f_{\Phi}(\sigma)$}

Let $(E,\Phi)$ be a CM-pair with $E$ a field. In (\ref{b57d}) we saw that
$N_{\Phi}$ gives a well-defined homomorphism $\Aut(\mathbb{C}{}/E^{\ast
})\rightarrow\mathbb{A}{}_{f,E}^{\times}/E^{\times}$. In this subsection, we
extend this to a homomorphism on the whole of $\Aut(\mathbb{C}{})$.

Choose an embedding $E\hookrightarrow\mathbb{C},$ and extend it to an
embedding $i\colon E^{\ab}\hookrightarrow\mathbb{C}$. Choose elements
$w_{\rho}\in\Aut(\mathbb{C})$, one for each $\rho\in\Hom(E,\mathbb{C})$, such
that
\[
w_{\rho}\circ i|E=\rho,\quad w_{\iota\rho}=\iota w_{\rho}.
\]
For example, choose $w_{\rho}$ for $\rho\in\Phi$ (or any other CM-type) to
satisfy the first equation, and then define $w_{\rho}$ for the remaining
$\rho$ by the second equation. For any $\sigma\in\Aut(\mathbb{C})$,
$w_{\sigma\rho}^{-1}\sigma w_{\rho}\circ i|E=w_{\sigma\rho}^{-1}\circ
\sigma\rho|E=i$. Thus $i^{-1}\circ w_{\sigma\rho}^{-1}\sigma w_{\rho}\circ
i\in\Gal(E^{\ab}/E)$, and we can define $F_{\Phi}:\Aut(\mathbb{C}%
)\rightarrow\Gal(E^{\ab}/E)$ by
\[
F_{\Phi}(\sigma)=\prod_{\varphi\in\Phi}i^{-1}\circ w_{\sigma\varphi}%
^{-1}\sigma w_{\varphi}\circ i.
\]

\begin{lemma}
\label{b76}The element $F_{\Phi}$ is independent of the choice of $\{w_{\rho
}\}$.
\end{lemma}

\begin{proof}
Any other choice is of the form $w_{\rho}^{\prime}=w_{\rho}h_{\rho}$,
$h_{\rho}\in\Aut(\mathbb{C}/iE)$. Thus $F_{\Phi}(\sigma)$ is changed by
$i^{-1}\circ(\prod_{\varphi\in\Phi}h_{\sigma\varphi}^{-1}h_{\varphi})\circ i$.
The conditions on $w$ and $w^{\prime}$ imply that $h_{\iota\rho}=h_{\rho}$,
and it follows that the inside product is $1$ because $\sigma$ permutes the
unordered pairs $\{\varphi,\iota\varphi\}$ and so $\prod_{\varphi\in\Phi
}h_{\varphi}=\prod_{\varphi\in\Phi}h_{\sigma\varphi}$.
\end{proof}

\begin{lemma}
\label{b77}The element $F_{\Phi}$ is independent of the choice of $i$ (and
$E\hookrightarrow\mathbb{C}$).
\end{lemma}

\begin{proof}
Any other choice is of the form $i^{\prime}=\sigma\circ i$, $\sigma
\in\Aut(\mathbb{C})$. Take $w_{\rho}^{\prime}=w_{\rho}\circ\sigma^{-1}$, and
then
\[
F_{\Phi}^{\prime}(\tau)=\prod i^{\prime-1}\circ(\sigma w_{\tau\varphi}%
^{-1}\tau w_{\varphi}\sigma^{-1})\circ i^{\prime}=F_{\Phi}(\tau).
\]

\end{proof}

Thus we can suppose $E\subset\mathbb{C}$ and ignore $i$; then
\[
F_{\Phi}(\sigma)=\prod_{\varphi\in\Phi}w_{\sigma\varphi}^{-1}\sigma
w_{\varphi}\mod\Aut(\mathbb{C}/E^{\mathrm{ab}})
\]
where the $w_{\rho}$ are elements of $\Aut(\mathbb{C})$ such that
\[
w_{\rho}|E=\rho,\quad w_{\iota\rho}=\iota w_{\rho}.
\]

\begin{proposition}
\label{b78}For any $\sigma\in\Aut(\mathbb{C})$, there is a unique $f_{\Phi
}(\sigma)\in\mathbb{A}_{f,E}^{\times}/E^{\times}$ such that

\begin{enumerate}
\item \ \textrm{art}$_{E}(f_{\Phi}(\sigma))=F_{\Phi}(\sigma)$;

\item \ $f_{\Phi}(\sigma)\cdot\iota f_{\Phi}(\sigma)=\chi(\sigma)E^{\times}$,
$\chi=\chi_{\text{\textrm{cyc}}}$.
\end{enumerate}
\end{proposition}

\begin{proof}
Since \textrm{art}$_{E}$ is surjective, there is an $f\in\mathbb{A}%
_{f,E}^{\times}/E^{\times}$ such that \textrm{art}$_{E}(f)=F_{\Phi}(\sigma)$.
We have
\begin{align*}
\mathrm{art}_{E}(f\cdot\iota f)  &  =\mathrm{art}_{E}(f)\cdot\mathrm{art}%
_{E}(\iota f)\\
&  =\mathrm{art}_{E}(f)\cdot\iota\mathrm{art}_{E}(f)\iota^{-1}\\
&  =F_{\Phi}(\sigma)\cdot F_{\iota\Phi}(\sigma)\\
&  =\mathrm{Ver}_{E/\mathbb{Q}}(\sigma),
\end{align*}
where $\mathrm{Ver}_{E/\mathbb{Q}}\colon\Gal(\mathbb{Q}^{\al}/\mathbb{Q}%
)^{\text{\textrm{ab}}}\rightarrow\Gal(\mathbb{Q}^{\al}/E)^{\text{\textrm{ab}}%
}$ is the transfer (Verlagerung) map. As $\mathrm{Ver}_{E/\mathbb{Q}}%
=$\textrm{art}$_{E}\circ\chi$, it follows that $f\cdot\iota f=\chi
(\sigma)E^{\times}$ modulo $\Ker($\textrm{art}$_{E})$. Lemma \ref{b79} shows
that $1+\iota$ acts bijectively on $\Ker($\textrm{art}$_{E})$, and so there is
a unique $a\in\Ker($\textrm{art}$_{E})$ such that $a\cdot\iota a=(f\cdot\iota
f/\chi(\sigma))E^{\times}$; we must take $f_{\Phi}(\sigma)=f/a$.
\end{proof}

\begin{remark}
\label{b80}The above definition of $f_{\Phi}(\sigma)$ is due to Tate. The
original definition, due to Langlands, was more direct but used the Weil group
(\cite{langlands1979c}, \S 5).
\end{remark}

\begin{proposition}
\label{b81}The maps $f_{\Phi}\colon\Aut(\mathbb{C})\rightarrow\mathbb{A}%
_{f,E}^{\times}/E^{\times}$ have the following properties:

\begin{enumerate}
\item \ $f_{\Phi}(\sigma\tau)=f_{\tau\Phi}(\sigma)\cdot f_{\Phi}(\tau)$;

\item \ $f_{\Phi(\tau^{-1}|E)}(\sigma)=\tau f_{\Phi}(\sigma)$ if $\tau E=E$;

\item \ $f_{\Phi}(\iota)=1$.
\end{enumerate}
\end{proposition}

\begin{proof}
Let $f=f_{\tau\Phi}(\sigma)\cdot f_{\Phi}(\tau)$. Then
\[
\mathrm{art}_{E}(f)=F_{\tau\Phi}(\sigma)\cdot F_{\Phi}(\tau)=\prod_{\varphi
\in\Phi}w_{\sigma\tau\varphi}^{-1}\cdot\sigma w_{\tau\varphi}\cdot
w_{\tau\varphi}^{-1}\cdot\tau w_{\varphi}=F_{\Phi}(\sigma\tau)
\]
and $f\cdot\iota f=\chi(\sigma)\chi(\tau)E^{\times}=\chi(\sigma\tau)E^{\times
}$. Thus $f$ satisfies the conditions that determine $f_{\Phi}(\sigma\tau)$.
This proves (a), and (b) and (c) can be proved similarly.
\end{proof}

Let $E^{\ast}$ be the reflex field for $(E,\Phi)$, so that $\Aut(\mathbb{C}%
/E^{\ast})=\{\sigma\in\Aut(\mathbb{C})\mid\sigma\Phi=\Phi\}$. Then
$\Phi\Aut(\mathbb{C}/E)\overset{\text{{\tiny def}}}{=}\bigcup
\nolimits_{\varphi\in\Phi}\varphi\cdot\Aut(\mathbb{C}/E)$ is stable under the
left action of $\Aut(\mathbb{C}/E^{\ast})$, and we write
\[
\Aut(\mathbb{C}/E)\Phi^{-1}=\bigcup\psi\cdot\Aut(\mathbb{C}/E^{\ast}%
)\qquad(\text{\textrm{disjoint union}}).
\]
The set $\Psi=\{\psi|E^{\ast}\}$ is a CM-type for $E^{\ast}$, and $(E^{\ast
},\Psi)$ is the reflex of $(E,\Phi)$. The map $a\mapsto\prod_{\psi\in\Psi}%
\psi(a)\colon E^{\ast}\rightarrow\mathbb{C}$ factors through $E$ and defines a
morphism of algebraic tori $N_{\Phi}\colon T^{E^{\ast}}\rightarrow T^{E}$. The
fundamental theorem of complex multiplication over the reflex field states the
following: let $\sigma\in\Aut(\mathbb{C}/E^{\ast})$, and let $a\in
\mathbb{A}_{f,E^{\ast}}^{\times}/E^{\ast\times}$ be such that $\mathrm{art}%
_{E^{\ast}}(a)=\sigma$; then (\ref{b72}) is true after $f$ has been replaced
by $N_{\Phi}(a)$ (see Theorem \ref{b58}; also \cite{shimura1971}, Theorem
5.15; the sign differences result from different conventions for the
reciprocity law and the actions of Galois groups). The next result shows that
this is in agreement with (\ref{b72}).

\begin{proposition}
\label{b82}For any $\sigma\in\Aut(\mathbb{C}/E^{\ast})$ and $a\in
\mathbb{A}_{f,E^{\ast}}^{\times}/E^{\ast\times}$ such that $\mathrm{art}%
_{E^{\ast}}(a)=\sigma|E^{\ast\mathrm{ab}}$, $N_{\Phi}(a)\in f_{\Phi}(\sigma)$.
\end{proposition}

\begin{proof}
Partition $\Phi$ into orbits, $\Phi=\cup_{j}\Phi_{j}$, for the left action of
$\Aut(\mathbb{C}/E^{\ast})$. Then $\Aut(\mathbb{C}/E)\Phi^{-1}=\bigcup
_{j}\Aut(\mathbb{C}/E)\Phi_{j}^{-1}$, and
\[
\Aut(\mathbb{C}/E)\Phi_{j}^{-1}=\Aut(\mathbb{C}/E)(\sigma_{j}^{-1}%
\Aut(\mathbb{C}/E^{\ast}))=(\Hom_{E}(L_{j},\mathbb{C})\circ\sigma_{j}%
^{-1})\Aut(\mathbb{C}/E^{\ast})
\]
where $\sigma_{j}$ is any element of $\Aut(\mathbb{C})$ such that $\sigma
_{j}|E\in\Phi_{j}$ and $L_{j}=(\sigma_{j}^{-1}E^{\ast})E$. Thus $N_{\Phi
}(a)=\prod b_{j}$, with $b_{j}=\Nm_{L_{j}/E}(\sigma_{j}^{-1}(a))$. Let
\[
F_{j}(\sigma)=\prod_{\varphi\in\Phi_{j}}w_{\sigma\varphi}^{-1}\sigma
w_{\varphi}\quad(\text{mod}\Aut(\mathbb{C}/E^{\text{\textrm{ab}}})).
\]
We begin by showing that $F_{j}(\sigma)=\mathrm{art}_{E}(b_{j})$. The basic
properties of Artin's reciprocity law show that%
\[
\begin{CD}
{\mathbb{A}}_{f,E}^{\times}@>{\textrm{injective}}>>
{\mathbb{A}}_{f,\sigma L_j}^{\times}@>{\sigma_j^{-1}}>>
{\mathbb{A}}_{f,L_j}^{\times}@>{\Nm_{L_j/K}}>>
{\mathbb{A}}_{f,K}^{\times}\\
@VV{\mathrm{art}_E}V@VV{\mathrm{art}_{\sigma L_j}}V@VV{\mathrm{art}_{L_j}}V@V{\mathrm{art}_K}VV\\
\Gal(E^{\text{ab}}/E)@>{V_{\sigma_jL_j/E}}>>
\sigma_j\Gal(L_j^{\text{ab}}/L_j)\sigma_j^{-1}@>{\ad\sigma_j^{-1}}>>
\Gal(L_j^{\text{ab}}/L_j)@>{\text{restriction}}>>\Gal(K^{\text{ab}}/K)
\end{CD}
\]
commutes. Therefore \textrm{art}$_{E}(b_{j})$ is the image of $\mathrm{art}%
_{E^{\ast}}(a)$ by the three maps in the bottom row of the diagram. Consider
$\{t_{\varphi}\mid t_{\varphi}=w_{\varphi}\sigma_{j}^{-1},\quad\varphi\in
\Phi_{j}\}$; this is a set of coset representatives for $\sigma_{j}%
\Aut(\mathbb{C}/L_{j})\sigma_{j}^{-1}$ in $\Aut(\mathbb{C}/E^{\ast})$, and so
$F_{j}(\sigma)=\prod_{\varphi\in\Phi_{j}}\sigma_{j}^{-1}t_{\sigma\varphi}%
^{-1}\sigma t_{\varphi}\sigma_{j}=\sigma_{j}^{-1}V(\sigma)\sigma
_{j}\mod\Aut(\mathbb{C}/E^{\text{\textrm{ab}}})$.

Thus \textrm{art}$_{E}(N_{\Phi}(a))=\prod$\textrm{art}$_{E}(b_{j})=\prod
F_{j}(\sigma)=F_{\Phi}(\sigma)$. As $N_{\Phi}(a)\cdot\iota N_{\Phi}(a)\in
\chi_{\mathrm{cyc}}(\sigma)E^{\times}$ (see \ref{b57e}), this shows that
$N_{\Phi}(a)\in f_{\Phi}(\sigma)$.
\end{proof}

\subsection{Proof of Theorem \ref{b74} up to a sequence of signs}

The variety $\sigma A$ has type $(E,\sigma\Phi)$ because $\sigma\Phi$
describes the action of $E$ on the tangent space to $\sigma A$ at zero. Choose
any $E$-linear isomorphism $\alpha\colon H_{1}(A,\mathbb{Q})\rightarrow
H_{1}(\sigma A,\mathbb{Q})$. Then
\[
V_{f}(A)\overset{\sigma}{\rightarrow}V_{f}(\sigma A)\overset{(\alpha
\otimes1)^{-1}}{\rightarrow}V_{f}(A)
\]
is an $\mathbb{A}_{f,E}$-linear isomorphism, and hence is multiplication by
some $g\in\mathbb{A}_{f,E}^{\times}$; thus
\[
(\alpha\otimes1)\circ g=\sigma.
\]

\begin{lemma}
\label{b83}For this $g$, we have
\[
(\alpha\psi)(\frac{\chi(\sigma)}{g\cdot\iota g}x,y)=(\sigma\psi)(x,y),\quad
\text{\textrm{all }}x,y\in V_{f}(\sigma A).
\]

\end{lemma}

\begin{proof}
By definition,
\begin{align*}
(\sigma\psi)(\sigma x,\sigma y)  &  =\sigma(\psi(x,y))\qquad x,y\in V_{f}(A)\\
(\alpha\psi)(\alpha x,\alpha y)  &  =\psi(x,y)\qquad x,y\in V_{f}(A).
\end{align*}
On replacing $x$ and $y$ by $gx$ and $gy$ in the second inequality, we find
that
\[
(\alpha\psi)(\sigma x,\sigma y)=\psi(gx,gy)=\psi((g\cdot\iota g)x,y).
\]
As $\sigma(\psi(x,y))=\chi(\sigma)\psi(x,y)=\psi(\chi(\sigma)x,y)$, the lemma
is now obvious.
\end{proof}

\begin{remark}
\label{b84}

\begin{enumerate}
\item \ On replacing $x$ and $y$ with $\alpha x$ and $\alpha y$ in
(\ref{b83}), we obtain the formula
\[
\psi(\frac{\chi(\sigma)}{g\cdot\iota g}x,y)=(\sigma\psi)(\alpha x,\alpha y).
\]

\item On taking $x,y\in H_{1}(A,\mathbb{Q})$ in (\ref{b83}), we can deduce
that $\chi_{\mathrm{cyc}}(\sigma)/g\cdot\iota g\in E^{\times}$; therefore
$g\cdot\iota g\equiv\chi_{\mathrm{cyc}}(\sigma)$ modulo $E^{\times}$.
\end{enumerate}
\end{remark}

The only choice involved in the definition of $g$ is that of $\alpha$, and
$\alpha$ is determined up to multiplication by an element of $E^{\times}$.
Thus the class of $g$ in $\mathbb{A}_{f,E}^{\times}/E^{\times}$ depends only
on $A$ and $\sigma$. In fact, it depends only on $(E,\Phi)$ and $\sigma$,
because any other abelian variety of type $(E,\Phi)$ is isogenous to $A$ and
leads to the same class $gE^{\times}$. We define $g_{\Phi}(\sigma)=gE^{\times
}\in\mathbb{A}_{f,E}^{\times}/E^{\times}$.

\begin{proposition}
\label{b85}The maps $g_{\Phi}:\Aut(\mathbb{C})\rightarrow\mathbb{A}%
_{f,E}^{\times}/E^{\times}$ have the following properties:

\begin{enumerate}
\item \ $g_{\Phi}(\sigma\tau)=g_{\tau\Phi}(\sigma)\cdot g_{\Phi}(\tau)$;

\item \ $g_{\Phi(\tau^{-1}|E)}(\sigma)=\tau g_{\Phi}(\sigma)$ if $\tau E=E$;

\item \ $g_{\Phi}(\iota)=1$;

\item \ $g_{\Phi}(\sigma)\cdot\iota g_{\Phi}(\sigma)=\chi_{\mathrm{cyc}%
}(\sigma)E^{\times}$.
\end{enumerate}
\end{proposition}

\begin{proof}
(a) Choose $E$-linear isomorphisms $\alpha\colon H_{1}(A,\mathbb{Q}%
)\rightarrow H_{1}(\tau A,\mathbb{Q})$ and $\beta\colon H_{1}(\tau
A,\mathbb{Q})\rightarrow H_{1}(\sigma\tau A,\mathbb{Q})$, and let
$g=(\alpha\otimes1)^{-1}\circ\tau$ and $g_{\tau}=(\beta\otimes1)^{-1}%
\circ\sigma$ so that $g$ and $g_{\sigma}$ represent $g_{\Phi}(\tau)$ and
$g_{\tau\Phi}(\sigma)$ respectively. Then
\[
(\beta\alpha)\otimes1\circ(g_{\tau}g)=(\beta\otimes1)\circ g_{\tau}%
\circ(\alpha\otimes1)\circ g=\sigma\tau,
\]
which shows that $g_{\tau}g$ represents $g_{\Phi}(\sigma\tau)$.

(b) If $(A,E\hookrightarrow\End(A)\otimes\mathbb{Q})$ has type $(E,\Phi)$,
then $(A,E\overset{\tau^{-1}}{\rightarrow}E\rightarrow\End(A)\otimes
\mathbb{Q})$ has type $(E,\Phi\tau^{-1})$. The formula in (b) can be proved by
transport of structure.

(c) Complex conjugation $\iota\colon A\rightarrow\iota A$ is a homeomorphism
(relative to the complex topology) and so induces an $E$-linear isomorphism
$\iota_{1}\colon H_{1}(A,\mathbb{Q})\rightarrow H_{1}(A,\mathbb{Q})$. The map
$\iota_{1}\otimes1\colon V_{f}(A)\rightarrow V_{f}(\iota A)$ is $\iota$ again,
and so on taking $\alpha=\iota_{1}$, we find that $g=1$.

(d) This was proved in (\ref{b84}d).
\end{proof}

Theorem (\ref{b74}) (hence also \ref{b72}) becomes true if $f_{\Phi}$ is
replaced by $g_{\Phi}$. Our task is to show that $f_{\Phi}=g_{\Phi}$. To this
end we set
\begin{equation}
e_{\Phi}(\sigma)=g_{\Phi}(\sigma)/f_{\Phi}(\sigma)\in\mathbb{A}_{f,E}^{\times
}/E^{\times}. \label{e75}%
\end{equation}

\begin{proposition}
\label{b86}The maps $e_{\Phi}:\Aut(\mathbb{C})\rightarrow\mathbb{A}%
_{f,E}^{\times}/E^{\times}$ have the following properties:

\begin{enumerate}
\item $e_{\Phi}(\sigma\tau)=e_{\tau\Phi}(\sigma)\cdot e_{\Phi}(\tau)$;

\item $e_{\Phi(\tau^{-1}|E)}(\sigma)=\tau e_{\Phi}(\sigma)$ if $\tau E=E$;

\item $e_{\Phi}(\iota)=1$;

\item $e_{\Phi}(\sigma)\cdot\iota_{E}e_{\Phi}(\sigma)=1$;

\item $e_{\Phi}(\sigma)=1$ if $\sigma\Phi=\Phi$.
\end{enumerate}
\end{proposition}

\begin{proof}
Statements (a), (b), and (c) follow from (a), (b), and (c) of (\ref{b81}) and
(\ref{b85}), and (d) follows from (\ref{b78}b) and (\ref{b85}d). The condition
$\sigma\Phi=\Phi$ in (e) means that $\sigma$ fixes the reflex field of
$(E,\Phi)$ and, as we observed in the preceding subsection, the fundamental
theorem is known to hold in that case, which means that $f_{\Phi}%
(\sigma)=g_{\Phi}(\sigma)$.
\end{proof}

\begin{proposition}
\label{b87}Let $F$ be the largest totally real subfield of $E$; then $e_{\Phi
}(\sigma)\in\mathbb{A}_{f,F}^{\times}/F^{\times}$ and $e_{\Phi}(\sigma)^{2}%
=1$; moreover, $e_{\Phi}(\sigma)$ depends only on the effect of $\sigma$ on
$E^{\ast}$, and is $1$ if $\sigma|E^{\ast}=\id$.
\end{proposition}

\begin{proof}
Recall that $\sigma$ fixes $E^{\ast}$ if and only if $\sigma\Phi=\Phi$, in
which case (\ref{b86}e) shows that $e_{\Phi}(\sigma)=1$. Replacing $\tau$ by
$\sigma^{-1}\tau$ in (a), we find that $e_{\Phi}(\tau)=e_{\Phi}(\sigma)$ if
$\tau\Phi=\sigma\Phi$, i.e., $e_{\Phi}(\sigma)$ depends only on the
restriction of $\sigma$ to the reflex field of $(E,\Phi)$. From (b) with
$\tau=\iota$, we find using $\iota\Phi=\Phi\iota_{E}$ that $e_{\iota\Phi
}(\sigma)=\iota e_{\Phi}(\sigma)$. Putting $\tau=\iota$ in (a) and using (c)
we find that $e_{\Phi}(\sigma\iota)=\iota e_{\Phi}(\sigma)$; putting
$\sigma=\iota$ in (a) and using (c) we find that $e_{\Phi}(\iota
\sigma)=e_{\Phi}(\sigma)$. Since $\iota\sigma$ and $\sigma\iota$ have the same
effect on $E^{\ast}$, we conclude $e_{\Phi}(\sigma)=\iota e_{\Phi}(\sigma)$.
Thus $e_{\Phi}(\sigma)\in(\mathbb{A}_{f,E}^{\times}/E^{\times})^{\langle
\iota_{E}\rangle}$, which equals $\mathbb{A}_{f,F}^{\times}/F^{\times}$ by
Hilbert's Theorem 90.\footnote{The cohomology sequence of the sequence of
$\Gal(E/F)$-modules%
\[
1\rightarrow E^{\times}\rightarrow\mathbb{A}_{f,E}^{\times}\rightarrow
\mathbb{A}_{f,E}^{\times}/E^{\times}\rightarrow1
\]
is%
\[
1\rightarrow F^{\times}\rightarrow\mathbb{A}_{f,F}^{\times}\rightarrow
(\mathbb{A}_{f,E}^{\times}/E^{\times})^{\Gal(E/F)}\rightarrow H^{1}%
(\Gal(E/F),E^{\times})=0
\]
} Finally, (d) shows that $e_{\Phi}(\sigma)^{2}=1$.
\end{proof}

\begin{corollary}
\label{b88}Part (a) of (\ref{b72}) is true; part (b) of (\ref{b72}) becomes
true when $f$ is replaced by $ef$ with $e\in\mathbb{A}_{f,F}^{\times}$,
$e^{2}=1$.
\end{corollary}

\begin{proof}
Let $e\in e_{\Phi}(\sigma)$. Then $e^{2}\in F^{\times}$ and, since an element
of $F^{\times}$ that is a square locally at all finite primes is a square
(\cite{milneCFT} VIII 1.1), we can correct $e$ to achieve $e^{2}=1$. Now
(\ref{b72}) is true with $f$ replaced by $ef$, but $e$ (being a unit) does not
affect part (a) of (\ref{b72}).
\end{proof}

It remains to show that:%
\begin{equation}
\text{for all CM-fields }E\text{ and CM-types }\Phi\text{ on }E\text{,
}e_{\Phi}=1. \label{e82}%
\end{equation}

\subsection{Completion of the proof}

As above, let $(E,\Phi)$ be a CM pair, and let $e_{\Phi}(\sigma)=g_{\Phi
}(\sigma)/f_{\Phi}(\sigma)$ be the associated element of $\mathbb{A}{}%
_{f,E}^{\times}/E^{\times}$. Then, as in (\ref{b87}, \ref{b88}),%
\[
e_{\Phi}(\sigma)\in\mu_{2}(\mathbb{A}_{f,F})/\mu_{2}(F),\quad\sigma
\in\Aut(\mathbb{C}{}).
\]

Let
\[
e\in\mu_{2}(\mathbb{A}_{f,F})\text{, }e=(e_{v})_{v}\text{, }e_{v}=\pm1\text{,
}v\text{ a finite prime of }F
\]
be a representative for $e_{\Phi}(\sigma)$. We have to show that the $e_{v}$'s
are all $-1$ or all $+1$. For this, it suffices, to show that for, for any
prime numbers $\ell_{1}$ and $\ell_{2}$, the image of $e_{\Phi}(\sigma)$ in
$\mu_{2}(F_{\ell_{1}}\times F_{\ell_{2}})/\mu_{2}(F)$ is trivial. Here
$F_{\ell}=F\otimes_{\mathbb{Q}{}}\mathbb{Q}{}_{\ell}$.

In addition to the properties (a--e) of (\ref{b86}), we need:

\begin{enumerate}
\item[(f)] let $E^{\prime}$ be a CM-field containing $E$, and let
$\Phi^{\prime}$ be the extension of $\Phi$ to $E^{\prime}$; then for any
$\sigma\in\Aut(\mathbb{C}{})$,
\begin{equation}
e_{\Phi}\left(  \sigma\right)  =e_{\Phi^{\prime}}\left(  \sigma\right)
\quad\text{(in }\mathbb{A}{}_{f,E^{\prime}}^{\times}/E^{\prime\times}\text{).}
\label{e81}%
\end{equation}

\end{enumerate}

\noindent To prove this, one notes that the same formula holds for each of
$f_{\Phi}$ and $g_{\Phi}$: if $A$ is of type $(E,\Phi)$ then $A^{\prime
}\overset{\text{{\tiny def}}}{=}A\otimes_{E}E^{\prime}$ is of type
$(E^{\prime},\Phi^{\prime})$. Here $A^{\prime}=A^{M}$ with
$M=\Hom_{E\text{-linear}}(E^{\prime},E)$ (cf. \ref{b41}).

Note that (f) shows that $e_{\Phi^{\prime}}=1\implies e_{\Phi}=1$, and so it
suffices (\ref{e82}) for $E$ Galois over $\mathbb{Q}{}$ (and contained in
$\mathbb{C}{}$).

We also need:

\begin{enumerate}
\item[(g)] denote by $[\Phi]$ the characteristic function of $\Phi
\subset\Hom(E,\mathbb{C}{})$; then%
\[
\sum\nolimits_{i}n_{i}[\Phi_{i}]=0\implies\prod\nolimits_{i}e_{\Phi_{i}%
}(\sigma)^{n_{i}}=1\text{ for all }\sigma\in\Aut(\mathbb{C}{}).
\]

\end{enumerate}

\noindent This is a consequence of Deligne's theorem that all Hodge classes on
abelian varieties are absolutely Hodge, which tells us that the results on
abelian varieties with complex multiplication proved above extend to
CM-motives. The CM-motives are classified by infinity types rather than
CM-types, and (g) just says that the $e$ attached to the trivial CM-motive is
$1$. This will be explained in the next chapter.

We make (d) (of \ref{b86}) and (g) more explicit. Recall that an infinity type
on $E$ is a function $\rho\colon\Hom(E,\mathbb{C}{})\rightarrow\mathbb{Z}{}$
that can be written as a finite sum of CM-types (see \S 4). Now (g) allows us
to define $e_{\rho}$ by linearity for $\rho$ an infinity type on $E$.
Moreover,%
\[
e_{2\rho}=e_{\rho}^{2}=0,
\]
so that $e_{\rho}$ depends only on the reduction modulo $2$ of $\rho$, which
can be regarded as a function%
\[
\bar{\rho}\colon\Hom(E,\mathbb{\mathbb{C}{}}{})\rightarrow\mathbb{Z}%
{}/2\mathbb{Z}{},
\]
such that either (weight $0$)
\begin{equation}
\bar{\rho}(\varphi)+\bar{\rho}(\iota\varphi)=0\text{ for all }\varphi
\label{e78}%
\end{equation}
or (weight $1$)
\[
\bar{\rho}(\varphi)+\bar{\rho}(\iota\varphi)=1\text{ for all }\varphi\text{.}%
\]

We now prove that $e_{\bar{\rho}}=1$ if $\bar{\rho}$ is of weight $0$. The
condition (\ref{e78}) means that $\bar{\rho}(\varphi)=\bar{\rho}(\iota
\varphi)$, and so $\bar{\rho}$ arises from a function $q\colon
\Hom(F,\mathbb{C}{})\rightarrow\mathbb{Z}{}/2\mathbb{Z}{}$:%
\[
\bar{\rho}(\varphi)=q(\varphi|F).
\]
We write $e_{q}=e_{\bar{\rho}}$. When $E$ is a subfield of $\mathbb{C}{}$
Galois over $\mathbb{Q}{}$, (b) implies that there exists an $e(\sigma)\in
\mu_{2}(\mathbb{A}_{f,F})/\mu_{2}(F)$ such that\footnote{For each
$\varphi\colon F\rightarrow\mathbb{C}{}$, choose an extension (also denoted
$\varphi$) of $\varphi$ to $E$. Then%
\[
\bar{\rho}=\sum\nolimits_{\varphi^{\prime}\colon E\rightarrow\mathbb{C}{}}%
\bar{\rho}(\varphi^{\prime})\varphi^{\prime}=\sum\nolimits_{\varphi\colon
F\rightarrow\mathbb{C}{}}q(\varphi)(\varphi+\iota\varphi)
\]
and so%
\[
e_{q}(\sigma)\overset{\text{{\tiny def}}}{=}e_{\bar{\rho}}(\sigma
)=\prod\nolimits_{\varphi}e_{(1+\iota)\varphi}(\sigma)^{q(\varphi)}%
=\prod\nolimits_{\varphi}\varphi^{-1}(e_{1+\iota}(\sigma))^{q(\varphi)}%
\]
--- we can take $e(\sigma)=e_{1+\iota}(\sigma)$.}%
\[
e_{q}(\sigma)=\prod\nolimits_{\varphi\colon F\rightarrow\mathbb{C}{}}%
\varphi^{-1}(e(\sigma))^{q(\varphi)}\text{, }\sigma\in\Aut(\mathbb{C}{}).
\]
Write $e(\sigma)=e^{F}(\sigma)$ to denote the dependence of $e$ on $F$. It
follows from (f), that for any totally real field $F^{\prime}$ containing
$F$,
\[
e^{F}(\sigma)=\Nm_{F^{\prime}/F}e^{F^{\prime}}(\sigma).
\]
There exists a totally real field $F^{\prime}$, quadratic over $F$, and such
that all primes of $F$ dividing $\ell_{1}$ or $\ell_{2}$ remain prime in
$F^{\prime}$. The norm maps $\mu_{2}(F_{2,\ell})\rightarrow\mu_{2}%
(F_{1},_{\ell})$ are zero for $\ell=\ell_{1},\ell_{2}$, and so $e^{F}(\sigma)$
projects to zero in $\mu_{2}(F_{\ell_{1}})\times\mu_{2}(F_{\ell_{2}})/\mu
_{2}(F)$. Therefore $e_{q}(\sigma)$ projects to zero in $\mu_{2}(F_{\ell_{1}%
}\times F_{\ell_{2}})/\mu_{2}(F)$. This being true for every pair $(\ell
_{1},\ell_{2})$, we have $e_{q}=1$.

We now complete the proof of (\ref{e82}). We know that $e_{\bar{\rho}}$
depends only on the weight of $\bar{\rho}$, and so, for $\Phi$ a CM-type,
$e_{\Phi}(\sigma)$ depends only on $\sigma$. In calculating $e_{\Phi}(\sigma
)$, we may take $E=\mathbb{Q}{}(\sqrt{-1}\dot{)}$ and $\Phi$ to be one of the
two CM-types on $\mathbb{Q}{}[\sqrt{-1}]$. We know (see \ref{b87}) that
$e_{\Phi}(\sigma)$ depends only on $\sigma|E^{\ast}=\mathbb{Q}{}[\sqrt{-1}]$.
But $e_{\Phi}(1)=1=e_{\Phi}(\iota)$ by (\ref{b86}c).

\begin{aside}
Throughout, should allow $E$ to be a CM-algebra. Should restate Theorem
\ref{b74} with $\mathbb{C}{}$ replaced by $\mathbb{Q}{}^{\mathrm{al}}$; then
replace $\mathbb{C}{}$ with $\mathbb{Q}{}^{\mathrm{al}}$ throughout the proof
(so $\sigma$ is an automorphism of $\mathbb{Q}{}^{\mathrm{al}}$ rather than
$\mathbb{C}{}$).
\end{aside}

{
\bibliographystyle{cbe}
\bibliography{../../../refs/refs}
}

\end{document}